\documentclass[11pt]{amsart}  

\usepackage{amsbsy,amsmath,amsthm,amssymb,color,verbatim,ulem}
\usepackage{fourier}
\usepackage{mathtools, mathdots}
\usepackage{soul}
\usepackage{cancel,shuffle}
\usepackage{stackrel}
\usepackage{fdsymbol}
\usepackage{multicol}
\usepackage[all]{xy}
\usepackage{xargs}

\usepackage{float}
\usepackage{bbm}
\usepackage{fullpage,cancel,hyperref,cleveref}
\usepackage{float}
\usepackage{wrapfig}
\usepackage[dvipsnames]{xcolor}
\usepackage{graphicx}
\usepackage{subcaption}
\usepackage{multicol}
\usepackage{multirow}
\usepackage{longtable}
\usepackage{makecell}
\usepackage{supertabular}
\usepackage{thm-restate}

\usepackage{tikz}
\definecolor{darkgreen}{cmyk}{.9,0,.9,.2}
\definecolor{midgray}{gray}{0.60}
\definecolor{lightgray}{gray}{0.90}
\definecolor{lmgray}{gray}{0.70}
\usepackage{colortbl}
\usepackage{standalone}
\usetikzlibrary{arrows.meta,calc,decorations.markings,math,arrows.meta}

\tikzset{->-/.style={decoration={
  markings,
  mark=at position #1 with {\arrow{>}}},postaction={decorate}}}
\tikzset{-<-/.style={decoration={
  markings,
  mark=at position #1 with {\arrow{<}}},postaction={decorate}}}

\usepackage{algorithm}
\usepackage{algpseudocode}
\usepackage{stackengine,wasysym}
\usepackage{lipsum}
\newcommand\m[1]{\begin{pmatrix}#1\end{pmatrix}}

\newtheorem{thm}{Theorem}[section] 
\newtheorem*{thm*}{Theorem}
\newtheorem{prop}[thm]{Proposition}
\newtheorem{lem}[thm]{Lemma} 
\newtheorem{cor}[thm]{Corollary}

\theoremstyle{definition}                 
\newtheorem{defn}[thm]{Definition}

\theoremstyle{remark}                      
\newtheorem{exmp}[thm]{Example}  
\newtheorem{rem}[thm]{Remark}

\numberwithin{equation}{section}

\newcommand{\calG}{\mathcal{G}}

\def\multiset#1#2{\ensuremath{\left(\kern-.3em\left(\genfrac{}{}{0pt}{}{#1}{#2}\right)\kern-.3em\right)}}

\def\multiset#1#2{\ensuremath{\left(\kern-.3em\left(\genfrac{}{}{0pt}{}{#1}{#2}\right)\kern-.3em\right)}}

\usepackage[sorting=nyt]{biblatex}
\appto{\bibsetup}{\sloppy}

\title{Higher $q$-Continued Fractions and Dimers on Band Graphs}
\author{Wonwoo Kang}
\address[W. Kang]{International Center for Mathematical Sciences, Institute of Mathematics and Informatics, Bulgarian Academy of Sciences, Acad. G. Bonchev Str., Bl. 8, Sofia 1113, Bulgaria}
\email{\textcolor{blue}{\href{mailto:wonk@math.bas.bg}{wonk@math.bas.bg}}}
\author{Yucong Lei}
\address[Y. Lei]{Department of Mathematics, University of Michigan, Ann Arbor, MI 48109}
\email{\textcolor{blue}{\href{mailto:leiyc@umich.edu}{leiyc@umich.edu}}}

\begin{document}
\begin{abstract} 
In this paper, we explore the theory of higher dimers on band graphs. First, we provide a combinatorial interpretation for the trace of the $q$-deformed higher continued fraction matrices by showing that with respect to a $q$-weighting on edges, the trace gives the dimer partition function on the set of good higher dimers, which generalizes the notion of good matchings. We also show that the set of good higher dimer covers form a distributive lattice with respect to face flips on square faces. Finally, we attempt to generalize the symmetry result on circular fence posets to the case of good higher dimers, by showing that the dimer partition on a certain family of band graphs are palindromic, in particular, through an approach fitting in the context of dimer theory.
\end{abstract}
\maketitle

\tableofcontents

\section{Introduction}

The theory of cluster algebras, introduced by Fomin and Zelevinsky~\cite{FZ2002}, has inspired a vast body of work dedicated to the combinatorial realization of cluster variables. For cluster algebras originating from surfaces, the seminal work of Musiker, Schiffler, and Williams~\cite{musiker2011positivity} established that expansion formulas for cluster variables, which consists of cluster variables and coefficient variables, can be explicitly computed as sums over perfect matchings of specific planar graphs known as snake graphs. This combinatorial framework shows a deep connection to the theory of $q$-rational numbers and $q$-continued fractions~\cite{CS20, MGO20}. By specializing cluster variables to 1 and coefficient variables to $q$, the expansion formulas for snake graphs yield $q$-analogues of continued fractions. Recently, Ovsienko~\cite{ovsienko2025qrationalsdimers} employed dimer models to derive these same $q$-continued fractions. However, this approach—utilizing the dimer partition function—results in a shift in the power of $q$ when compared to the classical rank generating function approach.

In recent years, the study of snake graph combinatorics has been generalized beyond simple perfect matchings. Specifically, Musiker, Ovenhouse, Schiffler, and Zhang~\cite{Musiker2023HigherDC} studied higher dimer covers on snake graphs---which we refer to as $m$-dimer covers---to define an $SL_{m+1}$ analog of continued fractions for all $m\geq 1$. These $SL_{m+1}$ analogs of continued fractions also has a $q$-deformed version~\cite{BOSZ26}. The connection to double dimer cover to super Teichm\"uller theory is given in~\cite{musiker2022double}. Given these developments, it is natural to investigate how these higher-order structures are manifested within the dimer model framework proposed by Ovsienko.

To investigate the structural properties of these $m$-dimer covers, we first prove that the set of $m$-dimer covers on a planar bipartite graph forms a distributive lattice. We emphasize that our proof technique departs from the classical approach established by Propp~\cite{proppDimerThy}. As an immediate corollary, we obtain that the $m$-dimer covers on a snake graph, which we term $m$-dimer covers, also possess a distributive lattice structure.

Furthermore, while much of the existing literature focuses on snake graphs, the extension to band graphs~\cite{Apruzzese25,musiker2011positivity,musiker2013bases} (and their closely related circular fence posets) remains an area of active development. Introduced to handle generalized arcs and closed curves, band graphs describe elements that are vital to the broader algebraic structure of the cluster algebra such as bangle bases~\cite{musiker2013bases}. The expansions of closed curves require the more specific framework of ``good matchings'' rather than ordinary perfect matchings. Together, snake and band graphs connect the curves on the surfaces into a combinatorial expression.

In this paper, we develop further connections between higher dimer cover partition functions on band graphs and matrices of $q$-deformed higher continued fractions, building upon Ovsienko's dimer model framework~\cite{ovsienko2025qrationalsdimers}. Our main result extends this connection to the setting of ``good'' dimer covers on band graphs. Specifically, we show that the partition function of ``good'' higher dimer covers is elegantly realized by the trace of the matrix product corresponding to higher $q$-continued fractions, up to multiplication by $q^a$. 

\begin{thm*}
    Let $$R_{m}(q)=\m{q^m & q^{m-1} & \cdots & q & 1 \\
    0 & q^{m-1} & \cdots & q & 1 \\
    \vdots & & \ddots & \vdots & \vdots \\
    0 & 0 & \cdots & q & 1 \\
    0 & 0 & \cdots & 0 & 1 } \quad \text{and} \quad L_{m}(q)=\m{q^m & 0 & \cdots & 0 & 0 \\
    q^m & q^{m-1} & \cdots & 0 & 0 \\
    \vdots & \vdots & \ddots & \vdots & \vdots \\
    q^m & q^{m-1} & \cdots & q & 0 \\
    q^m & q^{m-1} & \cdots & q & 1}.$$ Then, $\text{tr}(R_m^{a_1}(q)L_m^{a_2}(q)\cdots R_m^{a_{2n-1}}(q)L_m^{a_{2n}}(q))$ is the sum of good $m$-dimer covers on the band graph \\
    $\overline{\mathcal{G}}[a_1,a_2,...,a_{2n-1},a_{2n}]$ under the SW-weighting (to be defined in Section~\ref{subsec: prelim_higher_dimer}), multiplied by some $q^a$. 
\end{thm*}

While the rank generating functions of perfect matchings for single dimer covers on band graphs are known to be symmetric~\cite{fence_posets}, we demonstrate that this property fails for double dimer covers on general band graphs. To resolve this, we restrict our attention to a specific family of band graphs parameterized by an ``almost palindromic'' sequence of positive integers. For this class, we prove that the $m$-dimer partition function $\overline{\mathcal{Z}}_{\overline{\mathcal{G}},m}(q)$ is indeed palindromic. We provide two independent proofs: a concise algebraic proof leveraging the trace properties of matrices, and a purely combinatorial proof.

Lastly, we connect our dimer theoretical results to the broader framework of cluster algebras from surfaces~\cite{FZ2007, MMSV24}. Building on the combinatorial realization of $\mathbf{g}$-vectors in posets~\cite{banaian2024skein}, we translate these constructions into the language of snake graphs. By analyzing the flippable tiles within the maximal and minimal perfect matchings of a snake graph's distributive lattice, we provide a concrete geometric interpretation of the $\mathbf{g}$-vector components $\mathbf{a}_\gamma$, $\mathbf{b}_\gamma$, and the boundary vector $\mathbf{r}_\gamma$. Crucially, we characterize the exact power of $q$ that bridges the dimer partition function and the rank generating function by calculating the $q$-weight of the minimal matching $\hat{0}$ in the higher dimer cover $\mathcal{D}_m(G)$.

The paper is organized as follows. In Section~\ref{sec:background}, we recall the notions of snake graphs and band graphs, as well as the necessary background on $m$-dimer covers and Ovsienko's construction. In Section~\ref{sec:good_higher_dimer}, we study good higher dimer covers on band graphs and present our main theorem connecting their partition functions to the traces of $q$-deformed matrix products. Section~\ref{sec: palindromy} investigates the palindromic properties of these partition functions for almost palindromic sequences, providing both algebraic and combinatorial proofs. In Section~\ref{sec: connections_to_cluster}, we relate our dimer-theoretic results to the $\mathbf{g}$-vectors of cluster algebras. Finally, in Section~\ref{sec: future}, we conclude with a discussion of open questions and future directions, including the unimodality of rank generating functions.

\section*{Acknowledgements}

W.K was supported by the Bulgarian Science Fund with grant no KP-06- N92/5, the Bulgarian Ministry of Education and Science with grant DO1-239/10.12.2024, and the Simons Foundation with grant SFI-MPS-T-Institutes-00007697. Y.L. was supported by NSF DMS-2246570. The authors thank Amanda Burcroff, Gregg Musiker, and Nick Ovenhouse for helpful discussions. The second author also thanks David Speyer for helpful conversations on Section~\ref{sec: palindromy}.

\section{Background}\label{sec:background}

\subsection{Snake Graphs and Band Graphs}

Snake graphs serve as the fundamental combinatorial bridge between the geometry of bordered surfaces and the algebraic structure of cluster algebras. In the work of Musiker, Schiffler, and Williams~\cite{musiker2011positivity}, it was established that for any cluster algebra associated with a surface $\mathcal{S}$, every cluster variable $x_\gamma$ corresponding to an arc $\gamma$ (or a closed curve $\ell$) can be computed explicitly using a snake graph $\mathcal{G}_\gamma$ (or a band graph $\mathcal{G}_\ell$). The construction of $\mathcal{G}_\gamma$ is governed by the crossing sequence of the arc $\gamma$ with a fixed initial triangulation $\mathcal{T}$. Each intersection between $\gamma$ and an edge of $\mathcal{T}$ corresponds to a square tile in the snake graph, where the relative orientation of the triangles in $\mathcal{T}$ dictates the gluing pattern—either horizontal or vertical—of successive tiles.

By utilizing this combinatorial realization, the authors were able to resolve the positivity conjecture for cluster algebras of surface type and provide explicit expansion formulas for cluster variables. Beyond these initial applications, snake graphs have revealed deep connections across various algebraic structures. For instance, \c{C}anak\c{c}{\i} and Schiffler~\cite{CS20} established a formal correspondence between the combinatorial structure of snake graphs and the theory of $q$-continued fractions, further expanding the utility of these objects in number theory.

A \textbf{snake graph} $\mathcal{G}$ is a connected planar graph consisting of a sequence of square tiles $\{\mathcal{G}_1, \mathcal{G}_2, \dots, \mathcal{G}_n\}$ for $n \geq 1$. For each $i \in \{1, \dots, n-1\}$, the tiles $\mathcal{G}_i$ and $\mathcal{G}_{i+1}$ share exactly one edge $e_i$, referred to as an \textbf{interior edge}. Throughout the paper, we think of $\mathcal{G}$ as embedded in the plane $\mathbb{R}^2$ in such a way that each square tile has its sides aligned with the coordinate axes. To maintain the directed nature of the graph, this interior edge $e_i$ must be either the north edge of $\mathcal{G}_i$ (coinciding with the south edge of $\mathcal{G}_{i+1}$) or the east edge of $\mathcal{G}_i$ (coinciding with the west edge of $\mathcal{G}_{i+1}$). We refer the reader to the aforementioned literature for a more details of the properties of snake graphs.

Here, we state a definition of snake graph in terms of gluing "zigzag blocks," which will make it easier later to describe the edge weights on them.

\begin{defn}
    A \textbf{zigzag block} is a planar graph formed by gluing a sequence of square tiles in the way that even-indexed square tiles are glued to the east of the previous tile, and odd-indexed square tiles are glued to the north edge of the previous tile. We denote a zigzag block consisting of $k$ tiles as $\mathcal{B}_k$, where $k=1,2,3,...$. See Figure~\ref{fig:small_examples_of_zigzag_blocks}. We also place a straight arrow or a bent arrow in each tile that points from its previous tile and to its next tile, and we set the convention that the arrow in the first tile to point from its south edge, and the arrow in the last tile to be a straight arrow.
\end{defn}

\begin{figure}[h!]
    \centering
    \tikzset{every picture/.style={line width=0.75pt}} 

\begin{tikzpicture}[x=0.75pt,y=0.75pt,yscale=-1,xscale=1]

\draw   (23.27,86) -- (59.27,86) -- (59.27,122) -- (23.27,122) -- cycle ;
\draw   (135.27,86) -- (171.27,86) -- (171.27,122) -- (135.27,122) -- cycle ;
\draw   (207.27,86) -- (243.27,86) -- (243.27,122) -- (207.27,122) -- cycle ;
\draw   (243.27,86) -- (279.27,86) -- (279.27,122) -- (243.27,122) -- cycle ;
\draw   (243.27,50) -- (279.27,50) -- (279.27,86) -- (243.27,86) -- cycle ;
\draw [color={rgb, 255:red, 128; green, 128; blue, 128 }  ,draw opacity=1 ]   (41.27,118.99) -- (41.27,91.01) ;
\draw [shift={(41.27,89.01)}, rotate = 90] [color={rgb, 255:red, 128; green, 128; blue, 128 }  ,draw opacity=1 ][line width=0.75]    (10.93,-3.29) .. controls (6.95,-1.4) and (3.31,-0.3) .. (0,0) .. controls (3.31,0.3) and (6.95,1.4) .. (10.93,3.29)   ;
\draw   (99.27,86) -- (135.27,86) -- (135.27,122) -- (99.27,122) -- cycle ;
\draw [color={rgb, 255:red, 128; green, 128; blue, 128 }  ,draw opacity=1 ]   (117.27,120) -- (117.27,104) -- (133.94,104) ;
\draw [shift={(135.94,104)}, rotate = 180] [color={rgb, 255:red, 128; green, 128; blue, 128 }  ,draw opacity=1 ][line width=0.75]    (10.93,-3.29) .. controls (6.95,-1.4) and (3.31,-0.3) .. (0,0) .. controls (3.31,0.3) and (6.95,1.4) .. (10.93,3.29)   ;
\draw [color={rgb, 255:red, 128; green, 128; blue, 128 }  ,draw opacity=1 ]   (225.27,120) -- (225.27,104) -- (241.94,104) ;
\draw [shift={(243.94,104)}, rotate = 180] [color={rgb, 255:red, 128; green, 128; blue, 128 }  ,draw opacity=1 ][line width=0.75]    (10.93,-3.29) .. controls (6.95,-1.4) and (3.31,-0.3) .. (0,0) .. controls (3.31,0.3) and (6.95,1.4) .. (10.93,3.29)   ;
\draw [color={rgb, 255:red, 128; green, 128; blue, 128 }  ,draw opacity=1 ]   (137.98,104) -- (166.57,104) ;
\draw [shift={(168.57,104)}, rotate = 180] [color={rgb, 255:red, 128; green, 128; blue, 128 }  ,draw opacity=1 ][line width=0.75]    (10.93,-3.29) .. controls (6.95,-1.4) and (3.31,-0.3) .. (0,0) .. controls (3.31,0.3) and (6.95,1.4) .. (10.93,3.29)   ;
\draw [color={rgb, 255:red, 128; green, 128; blue, 128 }  ,draw opacity=1 ]   (261.27,82.99) -- (261.27,55.01) ;
\draw [shift={(261.27,53.01)}, rotate = 90] [color={rgb, 255:red, 128; green, 128; blue, 128 }  ,draw opacity=1 ][line width=0.75]    (10.93,-3.29) .. controls (6.95,-1.4) and (3.31,-0.3) .. (0,0) .. controls (3.31,0.3) and (6.95,1.4) .. (10.93,3.29)   ;
\draw [color={rgb, 255:red, 128; green, 128; blue, 128 }  ,draw opacity=1 ]   (244.94,104) -- (262.27,104) -- (262.27,88.67) ;
\draw [shift={(262.27,86.67)}, rotate = 90] [color={rgb, 255:red, 128; green, 128; blue, 128 }  ,draw opacity=1 ][line width=0.75]    (10.93,-3.29) .. controls (6.95,-1.4) and (3.31,-0.3) .. (0,0) .. controls (3.31,0.3) and (6.95,1.4) .. (10.93,3.29)   ;
\draw   (315.27,86) -- (351.27,86) -- (351.27,122) -- (315.27,122) -- cycle ;
\draw   (351.27,86) -- (387.27,86) -- (387.27,122) -- (351.27,122) -- cycle ;
\draw   (351.27,50) -- (387.27,50) -- (387.27,86) -- (351.27,86) -- cycle ;
\draw [color={rgb, 255:red, 128; green, 128; blue, 128 }  ,draw opacity=1 ]   (333.27,120) -- (333.27,104) -- (349.94,104) ;
\draw [shift={(351.94,104)}, rotate = 180] [color={rgb, 255:red, 128; green, 128; blue, 128 }  ,draw opacity=1 ][line width=0.75]    (10.93,-3.29) .. controls (6.95,-1.4) and (3.31,-0.3) .. (0,0) .. controls (3.31,0.3) and (6.95,1.4) .. (10.93,3.29)   ;
\draw [color={rgb, 255:red, 128; green, 128; blue, 128 }  ,draw opacity=1 ]   (352.94,104) -- (370.27,104) -- (370.27,88.67) ;
\draw [shift={(370.27,86.67)}, rotate = 90] [color={rgb, 255:red, 128; green, 128; blue, 128 }  ,draw opacity=1 ][line width=0.75]    (10.93,-3.29) .. controls (6.95,-1.4) and (3.31,-0.3) .. (0,0) .. controls (3.31,0.3) and (6.95,1.4) .. (10.93,3.29)   ;
\draw   (387.27,50) -- (423.27,50) -- (423.27,86) -- (387.27,86) -- cycle ;
\draw [color={rgb, 255:red, 128; green, 128; blue, 128 }  ,draw opacity=1 ]   (369.27,84) -- (369.27,68) -- (385.94,68) ;
\draw [shift={(387.94,68)}, rotate = 180] [color={rgb, 255:red, 128; green, 128; blue, 128 }  ,draw opacity=1 ][line width=0.75]    (10.93,-3.29) .. controls (6.95,-1.4) and (3.31,-0.3) .. (0,0) .. controls (3.31,0.3) and (6.95,1.4) .. (10.93,3.29)   ;
\draw [color={rgb, 255:red, 128; green, 128; blue, 128 }  ,draw opacity=1 ]   (389.98,68) -- (418.57,68) ;
\draw [shift={(420.57,68)}, rotate = 180] [color={rgb, 255:red, 128; green, 128; blue, 128 }  ,draw opacity=1 ][line width=0.75]    (10.93,-3.29) .. controls (6.95,-1.4) and (3.31,-0.3) .. (0,0) .. controls (3.31,0.3) and (6.95,1.4) .. (10.93,3.29)   ;
\draw   (451.27,89) -- (487.27,89) -- (487.27,125) -- (451.27,125) -- cycle ;
\draw   (487.27,89) -- (523.27,89) -- (523.27,125) -- (487.27,125) -- cycle ;
\draw   (487.27,53) -- (523.27,53) -- (523.27,89) -- (487.27,89) -- cycle ;
\draw [color={rgb, 255:red, 128; green, 128; blue, 128 }  ,draw opacity=1 ]   (469.27,123) -- (469.27,107) -- (485.94,107) ;
\draw [shift={(487.94,107)}, rotate = 180] [color={rgb, 255:red, 128; green, 128; blue, 128 }  ,draw opacity=1 ][line width=0.75]    (10.93,-3.29) .. controls (6.95,-1.4) and (3.31,-0.3) .. (0,0) .. controls (3.31,0.3) and (6.95,1.4) .. (10.93,3.29)   ;
\draw [color={rgb, 255:red, 128; green, 128; blue, 128 }  ,draw opacity=1 ]   (488.94,107) -- (506.27,107) -- (506.27,91.67) ;
\draw [shift={(506.27,89.67)}, rotate = 90] [color={rgb, 255:red, 128; green, 128; blue, 128 }  ,draw opacity=1 ][line width=0.75]    (10.93,-3.29) .. controls (6.95,-1.4) and (3.31,-0.3) .. (0,0) .. controls (3.31,0.3) and (6.95,1.4) .. (10.93,3.29)   ;
\draw   (523.27,53) -- (559.27,53) -- (559.27,89) -- (523.27,89) -- cycle ;
\draw [color={rgb, 255:red, 128; green, 128; blue, 128 }  ,draw opacity=1 ]   (505.27,87) -- (505.27,71) -- (521.94,71) ;
\draw [shift={(523.94,71)}, rotate = 180] [color={rgb, 255:red, 128; green, 128; blue, 128 }  ,draw opacity=1 ][line width=0.75]    (10.93,-3.29) .. controls (6.95,-1.4) and (3.31,-0.3) .. (0,0) .. controls (3.31,0.3) and (6.95,1.4) .. (10.93,3.29)   ;
\draw   (523.27,17) -- (559.27,17) -- (559.27,53) -- (523.27,53) -- cycle ;
\draw [color={rgb, 255:red, 128; green, 128; blue, 128 }  ,draw opacity=1 ]   (523.94,71) -- (541.27,71) -- (541.27,55.67) ;
\draw [shift={(541.27,53.67)}, rotate = 90] [color={rgb, 255:red, 128; green, 128; blue, 128 }  ,draw opacity=1 ][line width=0.75]    (10.93,-3.29) .. controls (6.95,-1.4) and (3.31,-0.3) .. (0,0) .. controls (3.31,0.3) and (6.95,1.4) .. (10.93,3.29)   ;
\draw [color={rgb, 255:red, 128; green, 128; blue, 128 }  ,draw opacity=1 ]   (541.27,49.99) -- (541.27,22.01) ;
\draw [shift={(541.27,20.01)}, rotate = 90] [color={rgb, 255:red, 128; green, 128; blue, 128 }  ,draw opacity=1 ][line width=0.75]    (10.93,-3.29) .. controls (6.95,-1.4) and (3.31,-0.3) .. (0,0) .. controls (3.31,0.3) and (6.95,1.4) .. (10.93,3.29)   ;

\draw (29,132.4) node [anchor=north west][inner sep=0.75pt]    {$\mathcal{B}_{1}$};
\draw (121,130.4) node [anchor=north west][inner sep=0.75pt]    {$\mathcal{B}_{2}$};
\draw (230,128.4) node [anchor=north west][inner sep=0.75pt]    {$\mathcal{B}_{3}$};
\draw (343,128.4) node [anchor=north west][inner sep=0.75pt]    {$\mathcal{B}_{4}$};
\draw (477,128.4) node [anchor=north west][inner sep=0.75pt]    {$\mathcal{B}_{5}$};
\draw (584,78.4) node [anchor=north west][inner sep=0.75pt]    {$...$};

\end{tikzpicture}
    \caption{Examples of zigzag blocks; the grey arrows are not part of the graphs, but are there to illustrate the direction in which the zigzag block grows}
    \label{fig:small_examples_of_zigzag_blocks}
    \end{figure}
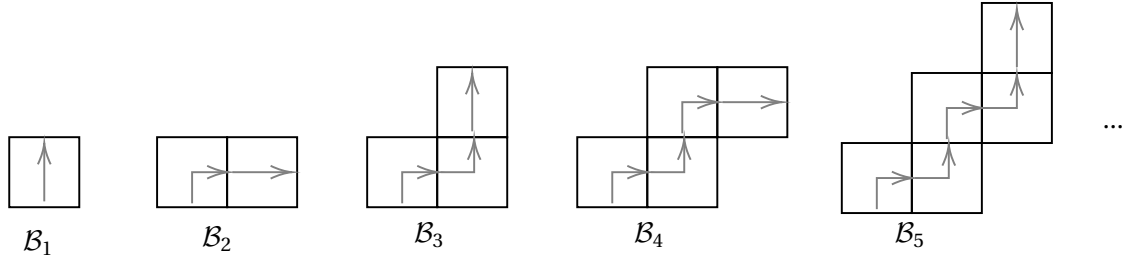

\begin{defn}\label{def: snake_graphs}
    Given a sequence of positive integers $a_1,a_2,...,a_n$, we define a \textbf{snake graph} by gluing the blocks $\mathcal{B}_{a_1}, \mathcal{B}_{a_2},...,\mathcal{B}_{a_n}$ or their images under reflection along any line with slope $1$: $\hat{\mathcal{B}}_{a_1}, \hat{\mathcal{B}}_{a_2},...,\hat{\mathcal{B}}_{a_n}$, in the unique way such that the arrow inside the first tile of $\mathcal{B}_{a_1}$ starts from its south edge, and the arrow in each tile after the gluing is aligned to either point to the north or to the east. We denote this snake graph as $\mathcal{G}[a_1,a_2,...,a_n+1]$. We label the left and right vertices of the south edge of the first tile to be $a,b$. 
\end{defn}

\begin{exmp}\label{exmp:snake_graph}
    In Figure~\ref{fig:gluing_blocks_to_snake}, we demonstrate how to obtain the snake graph. On the left, we position the blocks $\mathcal{B}_{a_1}, \mathcal{B}_{a_2}, \mathcal{B}_{a_3}, \mathcal{B}_{a_4}$ to align the arrows to point either north or east. Note the we need to do a reflection on $\mathcal{B}_{a_3}$ along lines of slope $1$ to get $\hat{\mathcal{B}}_{a_3}$. 
    \begin{figure}[h!]
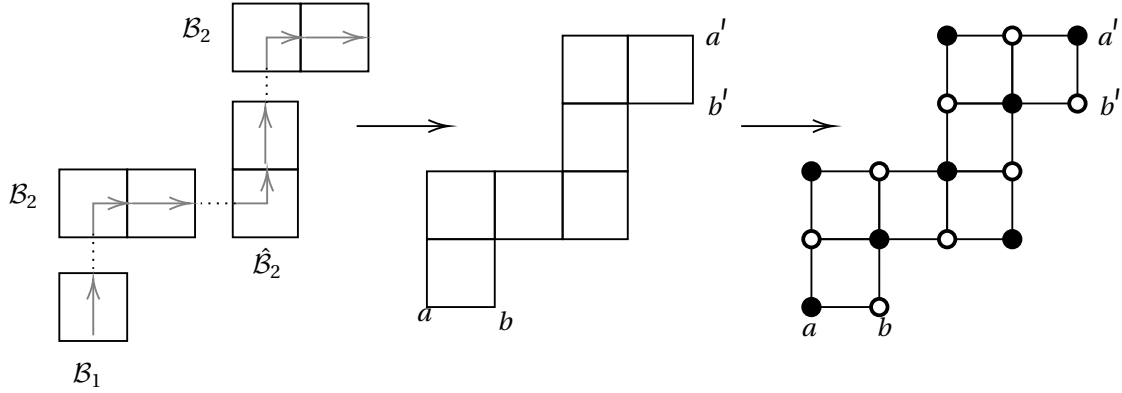

    \centering
    \includestandalone[mode=tex, width=0.9\textwidth]{figures/gluing_blocks_to_snake}
    \caption{Gluing blocks to form the snake graph $\mathcal{G}[1,2,2,3]$}
    \label{fig:gluing_blocks_to_snake}
    \end{figure}
\end{exmp}

We can treat snake graphs as induced subgraphs of the integer lattice. Thus, $\mathcal{G}[a_1,a_2,...,a_n]$ is naturally a bipartite graph. If $n$ is even, we say $\mathcal{G}$ is a snake graph of even length. We fix the convention to color the vertex $a$ to be black.


\begin{defn}\label{defn: even_length_band_graph}
    Take a snake graph of even length, $\mathcal{G}[a_1,a_2,...,a_{2n}+1]$, with $a_1,...,a_{2n}\geq 1$. Denote the black and white vertices of the unique edge of the last tile that the arrow inside the last tile points to as $a',b'$ (See Figure~\ref{fig:gluing_blocks_to_snake}).
    Glue $(ab)$ with $(a'b')$ along the directions $a\rightarrow b$ and $a'\rightarrow b'$. The resulting graph is called a \textbf{band graph}, and we denote this band graph constructed from $\calG[a_1,\ldots,a_{2n}+1]$ as $\overline{\mathcal{G}}[a_1,a_2,...,a_{2n-1},a_{2n}]$. 
\end{defn}




\begin{exmp}
    As shown in ~\ref{fig:gluing_blocks_to_snake}, the rightmost figure illustrates the snake graph $\calG[1,2,2,3]$ with labeled vertices $a,b$ in the south edge of the first tile, and $a',b'$ in the unique edge in the last tile the last arrow points to. Figure~\ref{fig:interior_exterior_faces} shows the corresponding band graph $\overline{\calG}[1,2,2,2]$.
\end{exmp}

\begin{rem}\label{rem: planarity_of_even_band}
    The band graph constructed by Definition~\ref{defn: even_length_band_graph} is naturally a planar bipartite graph, due to the fact that it has an even length of input integer sequence $(a_1,a_2,...,a_{2n})$.
\end{rem}

Note that more general band graphs lie on either an annulus or a M\"obius band~\cite{musiker2013bases}. A band graph can be constructed from a snake graph of arbitrary length; if the underlying snake graph has an even length, the resulting band graph lies on an annulus, whereas if it has an odd length, it lies on a M\"obius band.

\begin{rem}
    In the existing literature, a \textbf{sign function} $\varepsilon$ is used to characterize the relative orientation of successive tiles within a snake graph~\cite{CS20}. Every snake graph admits exactly two distinct sign functions which correspond to the two possible band graph constructions investigated in this work. Up to symmetry, a snake graph is uniquely determined by its sequence of tiles and the values of the sign function on its interior edges. In contrast to this traditional sign-based derivation, our construction uses arrows, providing a more direct way of construction of snake graph and band graphs.
\end{rem}

\subsection{Preliminaries on Higher Dimer Covers}\label{subsec: prelim_higher_dimer}

\begin{defn}
    A \textbf{planar bipartite graph} $G$ is a graph which is finite, connected, and has a planar embedding such that each bounded face is a disc, and there is a proper coloring of its vertices by black and white. Denote the unbounded face of $G$ as $f_{\infty}$, and the vertex set, edge set, and face set (including the unbounded face) as $V(G), E(G), F(G)$, respectively.
\end{defn}

\begin{defn}
Let $G$ be a planar bipartite graph. A \textbf{perfect matching} $P$ of $G$ is a subset of $E(G)$ such that every vertex of $G$ is incident to exactly one edge of $P$. Let $m\in\mathbb{N}_+$. An \textbf{$m$-dimer cover} $M$ of $G$
is a multiset of edges of $G$ whose multiplicity function
\[
    \operatorname{mult}_M:E(G)\longrightarrow\mathbb{Z}_{\geq 0}
\]
satisfies
\[
    \sum_{\substack{e\in E(G)\\e\ni v}}
    \operatorname{mult}_M(e)=m \text{ for all } v\in V(G)
\]
where $e\ni v$ means that $e$ is incident to $v$.

We denote the set of all $m$-dimer covers of $G$ by
$\mathcal{D}_m(G)$. When $m=1$, an $m$-dimer cover is naturally
identified with a perfect matching; under this identification,
\[
    \operatorname{mult}_P(e)=
    \begin{cases}
        1, & e\in P,\\
        0, & e\notin P.
    \end{cases}
\]
\end{defn}

\begin{rem}\label{rem: alternative_view_of_m_covers}
One can also think of an $m$-dimer cover $M$ as picking $m$ perfect matchings on $G$ and add up their edge multiplicity functions~\cite{musiker2022double}.
\end{rem}

\begin{defn}
    We say a graph $G$ has property $(*)_m$ if for every $e$, there are some $m$-dimer covers, $M_1,M_2$, such that $\text{mult}_{M_1}(e) >0, \text{mult}_{M_2}(e) = 0$. 
\end{defn}

\begin{defn}
    Given a band graph $\overline{G}= \overline{\mathcal{G}}[a_1,a_2,...,a_{2n-1},a_{2n}]$ whose underlying snake graph is $G = \mathcal{G}[a_1,a_2,...,a_{2n-1},a_{2n}+1]$. With respect to a planar embedding of $\overline{G}$, we define the \textbf{interior face} of $\overline{G}$ to be the face consisting of the image all the south facing and east facing boundary edges in $G$ under the gluing, and denote it as $f_0$. We also define the \textbf{exterior face} of $\overline{G}$ to be the face consisting of the image all the north facing and west facing boundary edges in $G$ under the gluing, and denote it as $f_\infty$. 
\end{defn}

\begin{figure}[h!]
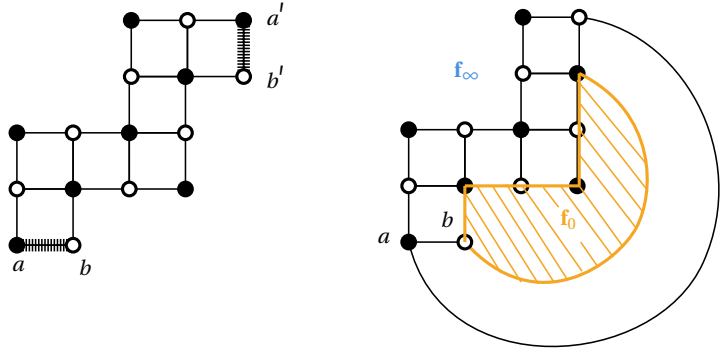

    \centering
    \includestandalone[mode=tex, width=0.6\textwidth]{figures/interior_exterior_faces}
    \caption{Left: The band graph $\overline{\mathcal{G}}[1,2,2,2]$ presented as a gluing of edges on $\mathcal{G}[1,2,2,3]$; Right: a planar embedding of $\overline{\mathcal{G}}[1,2,2,2]$, where the orange shaded face is the interior face, and the unbounded face is the exterior face}
    \label{fig:interior_exterior_faces}
    \end{figure}

\begin{rem}
    By Remark~\ref{rem: alternative_view_of_m_covers}, property $(*)_m$ is equivalent to property $(*)$ in ~\cite{dimer_knots}.
\end{rem}

\begin{lem}\label{lem: property_star_for_snakes}
    For any $a_1,...,a_{n}\geq 1$, the snake graph $G := \mathcal{G}[a_1,...,a_{n}+1]$ satisfies property $(*)_m$.
\end{lem}
\begin{proof}
    We use induction on the number of tiles , $a_1+a_2+...+a_n$. 
    
    When there is one single tile, the snake graph only has two perfect matchings $M_1, M_2$: $M_1$ contains the two horizontal edges of $G$ and $M_2$ contains the two vertical edges of $G$. Since $M_1\cup M_2 = E(G)$, each edge is contained in exactly one of $M_1$ and $M_2$.

    Now we suppose any snake graph with at most $N$ tiles satisfy property $(*)_m$. Consider a snake graph $G := \mathcal{G}[a_1,...,a_{n}+1]$ with $a_1+...+a_n = N+1$. Let $M$ be any perfect matching of $G$, then either it contains the edge $(a'b')$, or it contains both the edge incident to $a'$ and the incident to $b'$ other than $(a'b')$. In the first case, we have that $M\setminus\{(a'b')\}$ is a perfect matching on the induced subgraph $\mathcal{G}[a_1,a_2,...,a_n]$; in the second case, the remaining edges of $M$ excluding those two incident to $a',b'$ gives a perfect matching the subgraph $\mathcal{G}[a_1,...,a_{n-1}]$. Thus, any edge incident to $a'$ or $b'$ is contained in some matching but not some other matching based on the above two cases.
    
    Moreover, by inductive hypothesis, we know that $\mathcal{G}[a_1,a_2,...,a_n]$ satisfies property $(*)_m$. Thus, any edge of $G$ that is not incident to $a'$ or $b'$ is contained in some matching but not some other matching of $G$. 
\end{proof}

\begin{lem}\label{lem: property_star_for_planar_bands}
    Let $G := \overline{\mathcal{G}}[a_1,a_2,...,a_{2n-1}, a_{2n}]$ be a planar band graph, where $a_1,...,a_{2n}\geq 1$. Then $G$ satisfies property $(*)_m$.
\end{lem}

\begin{proof}
    Let $e_{a'},e_{b'}$ be the edges in the last tile of the corresponding snake graph $\mathcal{G}[a_1,a_2,...,a_{2n-1},a_{2n}+1]$ incident to $a',b'$, respectively that are not the edge $(a'b')$. Any perfect matching $M$ of $G$ that does not contain either $e_{a'}$ or $e_{b'}$ must contain the edges of a perfect matching of the subgraph $\mathcal{G}[a_1,a_2,...,a_{2n}]$ (See Figure~\ref{fig:property_star_on_band} on the right). Conversely, any perfect matching of the subgraph $\mathcal{G}[a_1,a_2,...,a_{2n}]$ gives a perfect matching of $G$ that does not contain $e_{a'},e_{b'}$. 

    \begin{figure}[h!]
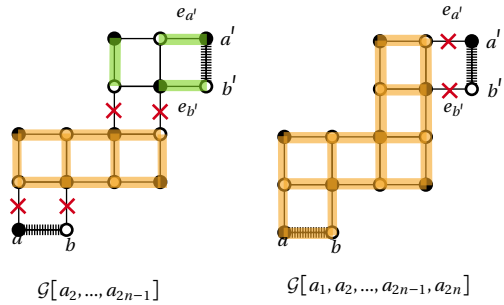

    \centering
    \includestandalone[mode=tex, width=0.4\textwidth]{figures/property_star_on_band}
    \caption{Different cases of perfect matchings on snake graphs}
    \label{fig:property_star_on_band}
    \end{figure}

    On the other hand, consider the case where $M$ contains both $e_{a'}, e_{b'}$. See Figure\ref{fig:property_star_on_band} on the left. Then, it must not contain the edge $(ab)$, and this forces $M$ to contain the set of edges of the unique perfect matching of the first block $\mathcal{B}_{a_1-1}$ not containing $(ab)$. Moreover, it also must contain the edges of the last block $\mathcal{B}_{a_{2n}}$ corresponding to the unique perfect matching containing both $e_{a'}, e_{b'}$. The remaining edges of $M$ should also form a perfect matching on the induced subgraph $\mathcal{G}[a_2,...,a_{2n-1}]$.

    Thus, by Lemma~\ref{lem: property_star_for_snakes}, any edge of the subgraph $\mathcal{G}[a_1,a_2,...,a_{2n-1},a_{2n}]$ is contained in some perfect matching of $G$ but not in some other; $e_{a'}, e_{b'}$ are contained in a perfect matching of $G$ of the second kind, but not the first kind.
\end{proof}
\begin{defn}
    Given a planar bipartite graph $G$ and a non-infinite face $f$ of $G$, with boundary edges $e_1,e_2,...,e_{2n}$ oriented clockwise, such that $e_{2i}$ are oriented from white to black, and $e_{2i-1}$ are oriented from black to white. Suppose $M$ is an $m$-dimer cover of $G$ such that $\text{mult}_{M}(e_{2i})\geq 1$ for all $i=1,...,n$, we can obtain another $m$-dimer cover $M'$ defined by $\text{mult}_{M'}(e_{2i-1}) = \text{mult}_{M}(e_{2i-1})+1$, $\text{mult}_{M'}(e_{2i}) = \text{mult}_{M}(e_{2i})-1$, and $\text{mult}_{M'}(e) = \text{mult}_{M}(e)$ for all other edges $e$, we call this operation an \textbf{up-flip}, and we define $M\lessdot M'$ when $M'$ is obtained from $M$ by an up-flip. Using up-flips as covering relations, one can define a poset $(\mathcal{D}_{m}(G), \leq)$, where $M\leq M'$ if there is a sequence of up-flips to transform $M$ into $M'$.
\end{defn}

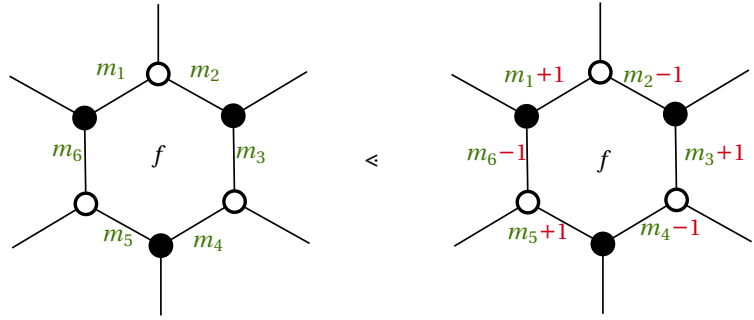
\begin{figure}[h!]
    \centering
    \tikzset{every picture/.style={line width=0.75pt}} 

\begin{tikzpicture}[x=0.75pt,y=0.75pt,yscale=-1,xscale=1]

\draw   (147.27,190.99) -- (103.15,166.41) -- (102.39,115.92) -- (145.73,90.01) -- (189.85,114.59) -- (190.61,165.08) -- cycle ;
\draw    (145.73,90.01) -- (145.73,48.93) ;
\draw    (147.27,232.07) -- (147.27,190.99) ;
\draw    (189.85,114.59) -- (233.19,88.68) ;
\draw    (59.81,192.32) -- (103.15,166.41) ;
\draw    (58.27,91.33) -- (102.39,115.92) ;
\draw    (190.61,165.08) -- (234.73,189.67) ;
\draw  [fill={rgb, 255:red, 255; green, 255; blue, 255 }  ,fill opacity=1 ][line width=1.5]  (139.46,90.01) .. controls (139.46,86.54) and (142.27,83.73) .. (145.73,83.73) .. controls (149.2,83.73) and (152.01,86.54) .. (152.01,90.01) .. controls (152.01,93.47) and (149.2,96.28) .. (145.73,96.28) .. controls (142.27,96.28) and (139.46,93.47) .. (139.46,90.01) -- cycle ;
\draw  [fill={rgb, 255:red, 255; green, 255; blue, 255 }  ,fill opacity=1 ][line width=1.5]  (184.34,165.08) .. controls (184.34,161.62) and (187.15,158.81) .. (190.61,158.81) .. controls (194.08,158.81) and (196.89,161.62) .. (196.89,165.08) .. controls (196.89,168.55) and (194.08,171.36) .. (190.61,171.36) .. controls (187.15,171.36) and (184.34,168.55) .. (184.34,165.08) -- cycle ;
\draw  [fill={rgb, 255:red, 255; green, 255; blue, 255 }  ,fill opacity=1 ][line width=1.5]  (96.88,166.41) .. controls (96.88,162.94) and (99.69,160.13) .. (103.15,160.13) .. controls (106.62,160.13) and (109.43,162.94) .. (109.43,166.41) .. controls (109.43,169.88) and (106.62,172.69) .. (103.15,172.69) .. controls (99.69,172.69) and (96.88,169.88) .. (96.88,166.41) -- cycle ;
\draw  [fill={rgb, 255:red, 0; green, 0; blue, 0 }  ,fill opacity=1 ][line width=1.5]  (183.57,114.59) .. controls (183.57,111.12) and (186.38,108.31) .. (189.85,108.31) .. controls (193.31,108.31) and (196.12,111.12) .. (196.12,114.59) .. controls (196.12,118.06) and (193.31,120.87) .. (189.85,120.87) .. controls (186.38,120.87) and (183.57,118.06) .. (183.57,114.59) -- cycle ;
\draw  [fill={rgb, 255:red, 0; green, 0; blue, 0 }  ,fill opacity=1 ][line width=1.5]  (96.11,115.92) .. controls (96.11,112.45) and (98.92,109.64) .. (102.39,109.64) .. controls (105.85,109.64) and (108.66,112.45) .. (108.66,115.92) .. controls (108.66,119.38) and (105.85,122.19) .. (102.39,122.19) .. controls (98.92,122.19) and (96.11,119.38) .. (96.11,115.92) -- cycle ;
\draw  [fill={rgb, 255:red, 0; green, 0; blue, 0 }  ,fill opacity=1 ][line width=1.5]  (140.99,190.99) .. controls (140.99,187.53) and (143.8,184.72) .. (147.27,184.72) .. controls (150.73,184.72) and (153.54,187.53) .. (153.54,190.99) .. controls (153.54,194.46) and (150.73,197.27) .. (147.27,197.27) .. controls (143.8,197.27) and (140.99,194.46) .. (140.99,190.99) -- cycle ;
\draw   (407.27,189.99) -- (363.15,165.41) -- (362.39,114.92) -- (405.73,89.01) -- (449.85,113.59) -- (450.61,164.08) -- cycle ;
\draw    (405.73,89.01) -- (405.73,47.93) ;
\draw    (407.27,231.07) -- (407.27,189.99) ;
\draw    (449.85,113.59) -- (493.19,87.68) ;
\draw    (319.81,191.32) -- (363.15,165.41) ;
\draw    (318.27,90.33) -- (362.39,114.92) ;
\draw    (450.61,164.08) -- (494.73,188.67) ;
\draw  [fill={rgb, 255:red, 255; green, 255; blue, 255 }  ,fill opacity=1 ][line width=1.5]  (399.46,89.01) .. controls (399.46,85.54) and (402.27,82.73) .. (405.73,82.73) .. controls (409.2,82.73) and (412.01,85.54) .. (412.01,89.01) .. controls (412.01,92.47) and (409.2,95.28) .. (405.73,95.28) .. controls (402.27,95.28) and (399.46,92.47) .. (399.46,89.01) -- cycle ;
\draw  [fill={rgb, 255:red, 255; green, 255; blue, 255 }  ,fill opacity=1 ][line width=1.5]  (444.34,164.08) .. controls (444.34,160.62) and (447.15,157.81) .. (450.61,157.81) .. controls (454.08,157.81) and (456.89,160.62) .. (456.89,164.08) .. controls (456.89,167.55) and (454.08,170.36) .. (450.61,170.36) .. controls (447.15,170.36) and (444.34,167.55) .. (444.34,164.08) -- cycle ;
\draw  [fill={rgb, 255:red, 255; green, 255; blue, 255 }  ,fill opacity=1 ][line width=1.5]  (356.88,165.41) .. controls (356.88,161.94) and (359.69,159.13) .. (363.15,159.13) .. controls (366.62,159.13) and (369.43,161.94) .. (369.43,165.41) .. controls (369.43,168.88) and (366.62,171.69) .. (363.15,171.69) .. controls (359.69,171.69) and (356.88,168.88) .. (356.88,165.41) -- cycle ;
\draw  [fill={rgb, 255:red, 0; green, 0; blue, 0 }  ,fill opacity=1 ][line width=1.5]  (443.57,113.59) .. controls (443.57,110.12) and (446.38,107.31) .. (449.85,107.31) .. controls (453.31,107.31) and (456.12,110.12) .. (456.12,113.59) .. controls (456.12,117.06) and (453.31,119.87) .. (449.85,119.87) .. controls (446.38,119.87) and (443.57,117.06) .. (443.57,113.59) -- cycle ;
\draw  [fill={rgb, 255:red, 0; green, 0; blue, 0 }  ,fill opacity=1 ][line width=1.5]  (356.11,114.92) .. controls (356.11,111.45) and (358.92,108.64) .. (362.39,108.64) .. controls (365.85,108.64) and (368.66,111.45) .. (368.66,114.92) .. controls (368.66,118.38) and (365.85,121.19) .. (362.39,121.19) .. controls (358.92,121.19) and (356.11,118.38) .. (356.11,114.92) -- cycle ;
\draw  [fill={rgb, 255:red, 0; green, 0; blue, 0 }  ,fill opacity=1 ][line width=1.5]  (400.99,189.99) .. controls (400.99,186.53) and (403.8,183.72) .. (407.27,183.72) .. controls (410.73,183.72) and (413.54,186.53) .. (413.54,189.99) .. controls (413.54,193.46) and (410.73,196.27) .. (407.27,196.27) .. controls (403.8,196.27) and (400.99,193.46) .. (400.99,189.99) -- cycle ;

\draw (107,82.4) node [anchor=north west][inner sep=0.75pt]    {$\textcolor[rgb]{0.25,0.46,0.02}{m_{1}}$};
\draw (162,83.4) node [anchor=north west][inner sep=0.75pt]    {$\textcolor[rgb]{0.25,0.46,0.02}{m}\textcolor[rgb]{0.25,0.46,0.02}{_{2}}$};
\draw (189,131.4) node [anchor=north west][inner sep=0.75pt]    {$\textcolor[rgb]{0.25,0.46,0.02}{m}\textcolor[rgb]{0.25,0.46,0.02}{_{3}}$};
\draw (164,183.4) node [anchor=north west][inner sep=0.75pt]    {$\textcolor[rgb]{0.25,0.46,0.02}{m}\textcolor[rgb]{0.25,0.46,0.02}{_{4}}$};
\draw (111,180.4) node [anchor=north west][inner sep=0.75pt]    {$\textcolor[rgb]{0.25,0.46,0.02}{m}\textcolor[rgb]{0.25,0.46,0.02}{_{5}}$};
\draw (81,130.4) node [anchor=north west][inner sep=0.75pt]    {$\textcolor[rgb]{0.25,0.46,0.02}{m}\textcolor[rgb]{0.25,0.46,0.02}{_{6}}$};
\draw (347,84.4) node [anchor=north west][inner sep=0.75pt]    {$\textcolor[rgb]{0.25,0.46,0.02}{m}\textcolor[rgb]{0.25,0.46,0.02}{_{1}}\textcolor[rgb]{0.82,0.01,0.11}{+1}$};
\draw (453,130.4) node [anchor=north west][inner sep=0.75pt]    {$\textcolor[rgb]{0.25,0.46,0.02}{m}\textcolor[rgb]{0.25,0.46,0.02}{_{3}}\textcolor[rgb]{0.82,0.01,0.11}{+1}$};
\draw (349,176.4) node [anchor=north west][inner sep=0.75pt]    {$\textcolor[rgb]{0.25,0.46,0.02}{m}\textcolor[rgb]{0.25,0.46,0.02}{_{5}}\textcolor[rgb]{0.82,0.01,0.11}{+1}$};
\draw (417,84.4) node [anchor=north west][inner sep=0.75pt]    {$\textcolor[rgb]{0.25,0.46,0.02}{m}\textcolor[rgb]{0.25,0.46,0.02}{_{2}}\textcolor[rgb]{0.82,0.01,0.11}{-1}$};
\draw (427,174.4) node [anchor=north west][inner sep=0.75pt]    {$\textcolor[rgb]{0.25,0.46,0.02}{m}\textcolor[rgb]{0.25,0.46,0.02}{_{4}}\textcolor[rgb]{0.82,0.01,0.11}{-1}$};
\draw (325,130.4) node [anchor=north west][inner sep=0.75pt]    {$\textcolor[rgb]{0.25,0.46,0.02}{m}\textcolor[rgb]{0.25,0.46,0.02}{_{6}}\textcolor[rgb]{0.82,0.01,0.11}{-1}$};
\draw (140,130.4) node [anchor=north west][inner sep=0.75pt]    {$f$};
\draw (402,133.4) node [anchor=north west][inner sep=0.75pt]    {$f$};
\draw (265,135.4) node [anchor=north west][inner sep=0.75pt]    {$\lessdot $};

\end{tikzpicture}
    \caption{An illustration of an up-flip around $f$}
    \label{fig:up_flip}
\end{figure}

The goal of the remainder of this section is to better understand this poset structure and in particular, show the following:
\begin{thm}\label{thm: dist_lattice_of_m}
    If $G$ has property $(*)_m$, then $(\mathcal{D}_m(G), \leq )$ is a distributive lattice. Call $(\mathcal{D}_m(G), \leq )$ the \textbf{$m$-dimer lattice} of $G$.
\end{thm}

Combining the above theorem with Lemma~\ref{lem: property_star_for_snakes}, we have the following:

\begin{cor}[special case of Theorem 8 in~\cite{claussen2025mixeddimermodelseuler}]
    The set of $m$-dimer covers on any snake graph forms a distributive lattice under the cover relation of up-flips.
\end{cor}

\begin{rem}
    One must be cautious not to derive the above theorem from Theorem 2 in ~\cite{proppDimerThy}. Theorem 2 in ~\cite{proppDimerThy} proves distributive lattice structures for ``$\mathbf{d}$-factors," which are induced subgraphs of $G$ whose degree at each vertex $v\in V(G)$ is $\mathbf{d}(v)$. However, "$\mathbf{d}$-factors" are different from $m$-dimer covers because it does not allow edges with multiplicities.
\end{rem}

We begin with some setups. Let $G$ be a planar bipartite graph with property $(*)_m$. Denote by $G^*$ its \textbf{dual graph}, defined by placing a vertex $f^*$ inside each face $f\in F(G)$, and for each edge $e\in E(G)$, connecting the vertices $f^*_1,f^*_2$ inside two neighboring faces $f_1,f_2$ by an undirected path that only intersects $e$ transversely. Denote this path $e^*$, and denote by $\Vec{e^*}$ an orientation of $e^*$.

We now define sign functions on oriented edges and oriented cycles in $G^*$. Let $C$ be a simple cycle in $G^*$, define $\text{sgn}(C) = +1$ if $C$ is oriented counterclockwise, and $\text{sgn}(C) = -1$ if $C$ is oriented clockwise. Let $e\in E(G)$, define $\text{sgn}(\Vec{e^*}) = +1$ if the white endpoint of $e$ is on the left of $\Vec{e^*}$, and $\text{sgn}(\Vec{e^*}) = -1$ if the white endpoint of $e$ is on the right of $\Vec{e^*}$. 

\begin{lem}\label{lem: edge_mult_equations}
    Let $M\in \mathcal{D}_{m}(G)$, $w = \text{mult}_M$. Let $C$ be any oriented cycle in $G^*$, consisting of oriented edges $\Vec{e_1^*},\Vec{e_2^*},...,\Vec{e_L^*}$. Let $W_C, B_C$ be the number of white, black vertices of $G$ enclosed by $C$. Then,
    \begin{equation}\label{eqn: signed_count_of_edges}
        \sum^L_{i=1}\text{sgn}(\Vec{e_i^*})w(e_i) = m\cdot\text{sgn}(C)\cdot(W_C-B_C)
    \end{equation}
\end{lem}
\begin{proof}
    WLOG assume $C$ is counterclockwise. If $C$ only encloses a white vertex, then that white vertex is on the left of all edges of $C$, so the LHS of ~\ref{eqn: signed_count_of_edges} becomes $\sum^L_{i=1}w(e_i) = m$. If $C$ only encloses a black vertex, then that black vertex is on the right of all edges of $C$, so the LHS of ~\ref{eqn: signed_count_of_edges}  becomes $\sum^L_{i=1}-w(e_i) = -m$. If $C$ encloses more than one vertex, we can break $C$ up into cycles around only one vertex by adding pairs of opposite orientations of some edges, which does not change the LHS of ~\ref{eqn: signed_count_of_edges} . See Figure~\ref{fig:decompose_cycles}.

    \begin{figure}[h!]
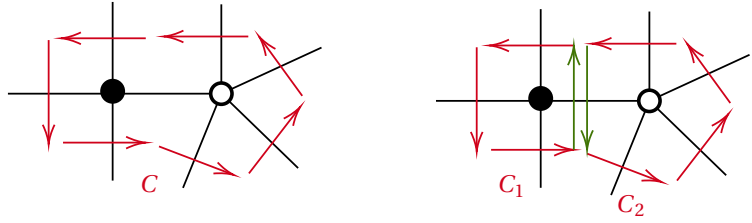

    \centering
    \includestandalone[mode=tex, width=0.6\textwidth]{figures/decompose_cycles}
    \caption{An decomposition of a cycle into smaller cycles enclosing single vertices}
    \label{fig:decompose_cycles}
    \end{figure}

\end{proof}

\begin{defn}
    Fix an $m$-dimer cover $M_0\in \mathcal{D}_{m}(G)$, which we call the \textbf{ground cover}. For any $M\in \mathcal{D}_{m}(G)$, we defined the \textbf{height function} of $M$ with respect to the gound cover as $h_M: F(G)\rightarrow \mathbb{Z}$, such that:
    \begin{enumerate}
        \item[(i)] $h_M(f_\infty) = 0$
        \item[(ii)] Let $f$ be any bounded face, and $p$ be any path in $G^*$ from $f_{\infty}$ to $f$, and $\Vec{e_1^*},\Vec{e_2^*},...,\Vec{e_T^*}$ be oriented edges of $p$. Then, $h_M(f) = \sum^T_{i=1}\text{sgn}(\Vec{e_i^*})(\text{mult}
        _M(e_i)-\text{mult}
        _{M_0}(e_i))$. 
    \end{enumerate}
\end{defn}
\begin{rem}
    Height functions are well-defined. In particular, the quantity $\sum^T_{i=1}\text{sgn}(\Vec{e_i^*})(\text{mult}
        _M(e_i)-\text{mult}
        _{M_0}(e_i))$ in (ii) above does not depend on the choice of the path $p$ from $f_\infty$ to $f$. This is because, by Lemma~\ref{lem: edge_mult_equations}, we have $$\sum^L_{i=1}\text{sgn}(\Vec{e_i^*})(\text{mult}
        _M(e_i)-\text{mult}
        _{M_0}(e_i)) = 0$$ for any oriented cycle $C$ consisting of the edges $\Vec{e_1^*},\Vec{e_2^*},...,\Vec{e_L^*}$.
\end{rem}

The following lemma characterizes when a function on $F(G)$ is a height function.
\begin{lem}\label{lem: inequalities_for_heights}
    Let $M_0$ be a ground cover. Let $h: F(G)\rightarrow \mathbb{Z}$ be any function with $h(f_{\infty})=0$. Then, there is a unique $M\in \mathcal{D}_{m}(G)$ such that $h = h_M$ with respect to the ground cover $M_0$ if and only if for any $e\in E(G)$,
    \begin{equation}\label{ineq: defining_conditions_for_heights}
        0\leq h(f_{R(e)})-h(f_{L(e)}) + \text{mult}_{M_0}(e)\leq m
    \end{equation}
    where $f_{L(e)}, f_{R(e)}$ are the left and right neighboring faces of $e$ when $e$ is oriented from black to white.  
\end{lem}
\begin{proof}
    To prove necessity, note that $h_M(f_{R(e)})-h_M(f_{L(e)}) + \text{mult}_{M_0}(e) = \text{mult}_M(e)\in [0,m]\cap \mathbb{Z}$.

    To prove sufficiency, define a function $w$ on edges by $w(e) := h_M(f_{R(e)})-h_M(f_{L(e)}) + \text{mult}_{M_0}(e)$. By assumption, $w(e)\in [0,m]\cap \mathbb{Z}$. We just need to show that the sum of $w(e)$ around any vertex $v\in V(G)$ is equal to $m$. We prove this for a white vertex $v$, and the proof for black vertices should be similar. Let $e_1,...,e_d$ be all edges incident to a white vertex $v\in V(G)$, oriented counterclockwise. Then we have $f_{R(e_i)} = f_{L(e_{i+1})}$ for all $i\in \mathbb{Z}_d$. Thus, 
    \begin{equation*}
        \begin{split}
            \sum^d_{i=1}w(e_i) &= \sum^d_{i=1}(h_M(f_{R(e_i)})-h_M(f_{L(e_i)}) + \text{mult}_{M_0}(e_i))\\
            &= \sum_{i\in \mathbb{Z}_d}(h_M(f_{L(e_{i+1})})-h_M(f_{L(e_i)}) + \text{mult}_{M_0}(e_i))=\sum_{i\in \mathbb{Z}_d}\text{mult}_{M_0}(e_i) = m\\
        \end{split}
    \end{equation*}
\end{proof}

\begin{lem}\label{lem:height_diff}
    Let $M\lessdot M'$ be two $m$-dimer covers related by an up-flip around face $\overline{f}$. Then \[h_{M'}(f) - h_{M}(f) = \mathbbm{1}_{\{f=\overline{f}\}}(f).\]
\end{lem}
\begin{proof}
    Let $e_1,e_2,...,e_{2n}$ be boundary edges of $f$, oriented clockwise, such that $e_{2i}$ are oriented from white to black, and $e_{2i-1}$ are oriented from black to white. Note that the induced subgraph of $G^*$ by deleting the vertex $\overline{f}^*$ and the edges $e^*_1,e^*_2,...,e^*_{2n}$ is connected, because deleting the face $\overline{f}$ is the same as puncturing a sphere, which is still path-connected, and if we take a path in the punctured sphere, from  $f_{\infty}$ to $f\neq \overline{f}$ that does not intersect vertices of $G$, then the edges this path crosses give us a walk along the edges of the induced subgraph of $G^*$.

    Since we can always pick a path from $f_{\infty}^*$ to $f^*\neq \overline{f}^*$ that does not use the edges $e^*_1,e^*_2,...,e^*_{2n}$, we get $h_{M'}(f) =h_{M}(f)$ for $f\neq \overline{f}$, because $\text{mult}_M(e) = \text{mult}_{M'}(e)$ for all $e\neq e_1,...,e_{2n}$. On the other hand, let $f_1$ be the face adjacent to $\overline{f}$ across from $e_1$. Since $\text{mult}_{M'}(e_{1}) = \text{mult}_{M}(e_{1})+1$, we have $h_{M'}(\overline{f}) -t_{M}(\overline{f}) = 1$, by computing $h_{M'}(\overline{f})$ through a path from $f_{\infty}$ to $f_1$ not using $e^*_1,e^*_2,...,e^*_{2n}$ appended with $e^*_1$.
\end{proof}

To prove that $(\mathcal{D}_m(G),\leq)$ is a lattice, we need to define meets and joins. We do this first on the level of height functions.

\begin{cor}
    Let $M_1,M_2\in \mathcal{D}_m(G)$. Then, there exists unique $m$-dimer covers, $M,M'\in \mathcal{D}_m(G)$, such that $h_M = \max\{h_{M_1}, h_{M_2}\}$ and $h_{M'} = \min\{h_{M_1}, h_{M_2}\}$. We denote $M$ by $M_1\vee M_2$, and $M'$ by $M_1\wedge M_2$. 
\end{cor}
\begin{proof}
    Note that 
    \begin{equation}
        \begin{split}
            &\min\{h_{M_1}(f_{R(e)}) - h_{M_1}(f_{L(e)}), h_{M_2}(f_{R(e)}) - h_{M_2}(f_{L(e)})\}\\
            &\leq (\max\{h_{M_1}, h_{M_2}\})(f_{R(e)}) - (\max\{h_{M_1}, h_{M_2}\})(f_{L(e)})\\
            &\leq \max\{h_{M_1}(f_{R(e)}) - h_{M_1}(f_{L(e)}), h_{M_2}(f_{R(e)}) - h_{M_2}(f_{L(e)})\}\\
        \end{split}
    \end{equation}
    and 
    \begin{equation}
        \begin{split}
            &\min\{h_{M_1}(f_{R(e)}) - h_{M_1}(f_{L(e)}), h_{M_2}(f_{R(e)}) - h_{M_2}(f_{L(e)})\}\\
            &\leq (\min\{h_{M_1}, h_{M_2}\})(f_{R(e)}) - (\min\{h_{M_1}, h_{M_2}\})(f_{L(e)})\\
            &\leq \max\{h_{M_1}(f_{R(e)}) - h_{M_1}(f_{L(e)}), h_{M_2}(f_{R(e)}) - h_{M_2}(f_{L(e)})\}\\
        \end{split}
    \end{equation}
    Thus, both $\min\{h_{M_1}, h_{M_2}\}$ and $\max\{h_{M_1}, h_{M_2}\}$ satisfy ~\ref{ineq: defining_conditions_for_heights}. By Lemma~\ref{lem: inequalities_for_heights}, we can find unique such $M,M'\in \mathcal{D}_m(G)$ as desired.
\end{proof}

So far, we proved that the potential candidates of join and meet of $M_1, M_2$ exist. We must also show they are indeed the unique smallest (or unique largest, rsp.) $m$-dimer covers larger than (or smaller than) $M_1,M_2$ with respect to the up-flip poset. To achieve this goal, we use the following lemma.
\begin{lem}\label{lem: compatibility_with_height_posets}
    Let $M_1, M_2\in \mathcal{D}_m(G)$. Suppose that for all faces $f\in F(G)$, we have $h_{M_1}(f)\leq h_{M_2}(f)$, and $\max_{f\in F(G)}\{h_{M_2}(f)-h_{M_1}(f)\}\geq 1$. Then, there exists a $\overline{f}\in F(G)$, such that one can perform an up-flip on $M_1$ around $f_0$ to obtain $M'_1$, and for all $f\in F(G)$, we have $h_{M'_1}(f)\leq h_{M_2}(f)$. In particular, if for all $f\in F(G), h_{M_1}(f)\leq h_{M_2}(f)$, then $M_1\leq M_2$.
\end{lem}
\begin{proof}
    The goal is to find a face $\overline{f}$ with $h_{M_1}(\overline{f}) < h_{M_2}(\overline{f})$, such that for any boundary edge $e$ of $f_0$ oriented white to black along the clockwise direction, we have $\text{mult}_{M_1}(e)\geq 1$. 

    We prove by contradiction. Suppose for any face $f$ with $h_{M_1}(f) < h_{M_2}(f)$, there exists some clockwise white-to-black boundary edge $e$ of $f$, such that $\text{mult}_{M_1}(e) = 0$. 
    
    We start with a face $f_1$ such that $h_{M_2}(f_1)-h_{M_1}(f_1) = \max_{f\in F(G)}\{h_{M_2}(f)-h_{M_1}(f)\}$. Then, we have a clockwise white-to-black boundary edge $e_1$ of $f_1$, such that $\text{mult}_{M_1}(e_1) = 0$. By definition of the height function, we have
    \begin{equation}
        \text{mult}_{M_1}(e_1) = \text{mult}_{M_2}(e_1) + (h_{M_2}-h_{M_1})(f_{R(e_1)}) - (h_{M_2}-h_{M_1})(f_{L(e_1)})
    \end{equation}
    where $f_{R(e_1)} = f_1$. This implies that $\text{mult}_{M_2}(e_1) = 0$, and $(h_{M_2}-h_{M_1})(f_{L(e_1)}) = (h_{M_2}-h_{M_1})(f_1) = \max_{f\in F(G)}\{h_{M_2}(f)-h_{M_1}(f)\}\geq 1$. Let $f_2 = f_{L(e_1)}$, $\Vec{e^*_i}$ be pointing from $f_1$ to $f_2$, and repeat the above process to find $f_3,f_4,...$ and $e_2,e_3,...$ where $(h_{M_2}-h_{M_1})(f_{i}) = \max_{f\in F(G)}\{h_{M_2}(f)-h_{M_1}(f)\}$, and $\text{mult}_{M_2}(e_i) = \text{mult}_{M_1}(e_i) = 0$ for all $i=1,2,...$.

    Since our graph $G$ has only finitely many faces, we must have $1\leq i < j$, where $f_i = f_j$. WLOG assume the path $f_i\rightarrow f_{i+1}\rightarrow...\rightarrow f_j$ consisting of $\Vec{e^*_i}, \Vec{e^*_{i+1}},...,\Vec{e^*_{j-1}}$ is a simple cycle, by choosing $f_i$ to be the first face visited twice. Denote this cycle by $C$. Then, by Lemma~\ref{lem: edge_mult_equations}, we must have $W_C = B_C$. Moreover, by construction, either all of the white endpoints of $e_{i},e_{i+1},...,e_{j-1}$ are enclosed by $C$ or all of their black endpoints are enclosed by $C$. WLOG suppose the black vertices of $e_{i},e_{i+1},...,e_{j-1}$ are inside $C$. See Figure~\ref{fig:distributive_pf_contradiction} for an illustration. 

    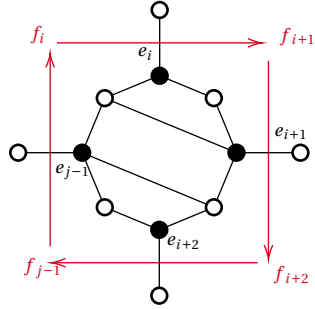
\begin{figure}[h!]
    \centering
    \tikzset{every picture/.style={line width=0.75pt}} 

\begin{tikzpicture}[x=0.75pt,y=0.75pt,yscale=-1,xscale=1]

\draw    (138.22,33.67) -- (138.22,83.33) ;
\draw  [fill={rgb, 255:red, 0; green, 0; blue, 0 }  ,fill opacity=1 ][line width=1.5]  (132.3,83.33) .. controls (132.3,80.06) and (134.95,77.41) .. (138.22,77.41) .. controls (141.49,77.41) and (144.14,80.06) .. (144.14,83.33) .. controls (144.14,86.6) and (141.49,89.26) .. (138.22,89.26) .. controls (134.95,89.26) and (132.3,86.6) .. (132.3,83.33) -- cycle ;
\draw    (96.78,100.21) -- (196.16,141.37) ;
\draw    (31.43,141.37) -- (79.73,141.37) ;
\draw    (179.11,182.54) -- (79.73,141.37) ;
\draw [color={rgb, 255:red, 208; green, 2; blue, 27 }  ,draw opacity=1 ]   (212.19,224.42) -- (58.19,224.42) ;
\draw [shift={(56.19,224.42)}, rotate = 360] [color={rgb, 255:red, 208; green, 2; blue, 27 }  ,draw opacity=1 ][line width=0.75]    (10.93,-3.29) .. controls (6.95,-1.4) and (3.31,-0.3) .. (0,0) .. controls (3.31,0.3) and (6.95,1.4) .. (10.93,3.29)   ;
\draw [color={rgb, 255:red, 208; green, 2; blue, 27 }  ,draw opacity=1 ]   (55.58,212.05) -- (55.58,68.05) ;
\draw [shift={(55.58,66.05)}, rotate = 90] [color={rgb, 255:red, 208; green, 2; blue, 27 }  ,draw opacity=1 ][line width=0.75]    (10.93,-3.29) .. controls (6.95,-1.4) and (3.31,-0.3) .. (0,0) .. controls (3.31,0.3) and (6.95,1.4) .. (10.93,3.29)   ;
\draw [color={rgb, 255:red, 208; green, 2; blue, 27 }  ,draw opacity=1 ]   (213.72,58.5) -- (60.72,58.5) ;
\draw [shift={(215.72,58.5)}, rotate = 180] [color={rgb, 255:red, 208; green, 2; blue, 27 }  ,draw opacity=1 ][line width=0.75]    (10.93,-3.29) .. controls (6.95,-1.4) and (3.31,-0.3) .. (0,0) .. controls (3.31,0.3) and (6.95,1.4) .. (10.93,3.29)   ;
\draw [color={rgb, 255:red, 208; green, 2; blue, 27 }  ,draw opacity=1 ]   (220.31,72.05) -- (220.31,220.05) ;
\draw [shift={(220.31,222.05)}, rotate = 270] [color={rgb, 255:red, 208; green, 2; blue, 27 }  ,draw opacity=1 ][line width=0.75]    (10.93,-3.29) .. controls (6.95,-1.4) and (3.31,-0.3) .. (0,0) .. controls (3.31,0.3) and (6.95,1.4) .. (10.93,3.29)   ;
\draw  [fill={rgb, 255:red, 255; green, 255; blue, 255 }  ,fill opacity=1 ][line width=1.5]  (132.3,33.67) .. controls (132.3,30.4) and (134.95,27.74) .. (138.22,27.74) .. controls (141.49,27.74) and (144.14,30.4) .. (144.14,33.67) .. controls (144.14,36.94) and (141.49,39.59) .. (138.22,39.59) .. controls (134.95,39.59) and (132.3,36.94) .. (132.3,33.67) -- cycle ;
\draw  [fill={rgb, 255:red, 0; green, 0; blue, 0 }  ,fill opacity=1 ][line width=1.5]  (73.81,141.37) .. controls (73.81,138.1) and (76.46,135.45) .. (79.73,135.45) .. controls (83,135.45) and (85.65,138.1) .. (85.65,141.37) .. controls (85.65,144.64) and (83,147.3) .. (79.73,147.3) .. controls (76.46,147.3) and (73.81,144.64) .. (73.81,141.37) -- cycle ;
\draw  [fill={rgb, 255:red, 0; green, 0; blue, 0 }  ,fill opacity=1 ][line width=1.5]  (190.24,141.37) .. controls (190.24,138.1) and (192.89,135.45) .. (196.16,135.45) .. controls (199.43,135.45) and (202.08,138.1) .. (202.08,141.37) .. controls (202.08,144.64) and (199.43,147.3) .. (196.16,147.3) .. controls (192.89,147.3) and (190.24,144.64) .. (190.24,141.37) -- cycle ;
\draw   (196.16,141.37) -- (179.11,182.54) -- (137.94,199.59) -- (96.78,182.54) -- (79.73,141.37) -- (96.78,100.21) -- (137.94,83.16) -- (179.11,100.21) -- cycle ;
\draw  [fill={rgb, 255:red, 255; green, 255; blue, 255 }  ,fill opacity=1 ][line width=1.5]  (25.51,141.37) .. controls (25.51,138.1) and (28.16,135.45) .. (31.43,135.45) .. controls (34.7,135.45) and (37.35,138.1) .. (37.35,141.37) .. controls (37.35,144.64) and (34.7,147.3) .. (31.43,147.3) .. controls (28.16,147.3) and (25.51,144.64) .. (25.51,141.37) -- cycle ;
\draw    (196.16,141.37) -- (244.46,141.37) ;
\draw  [fill={rgb, 255:red, 255; green, 255; blue, 255 }  ,fill opacity=1 ][line width=1.5]  (238.53,141.37) .. controls (238.53,138.1) and (241.19,135.45) .. (244.46,135.45) .. controls (247.73,135.45) and (250.38,138.1) .. (250.38,141.37) .. controls (250.38,144.64) and (247.73,147.3) .. (244.46,147.3) .. controls (241.19,147.3) and (238.53,144.64) .. (238.53,141.37) -- cycle ;
\draw    (137.94,199.59) -- (137.94,249.26) ;
\draw  [fill={rgb, 255:red, 255; green, 255; blue, 255 }  ,fill opacity=1 ][line width=1.5]  (132.02,249.26) .. controls (132.02,245.98) and (134.67,243.33) .. (137.94,243.33) .. controls (141.22,243.33) and (143.87,245.98) .. (143.87,249.26) .. controls (143.87,252.53) and (141.22,255.18) .. (137.94,255.18) .. controls (134.67,255.18) and (132.02,252.53) .. (132.02,249.26) -- cycle ;
\draw  [fill={rgb, 255:red, 0; green, 0; blue, 0 }  ,fill opacity=1 ][line width=1.5]  (132.02,199.59) .. controls (132.02,196.32) and (134.67,193.67) .. (137.94,193.67) .. controls (141.22,193.67) and (143.87,196.32) .. (143.87,199.59) .. controls (143.87,202.86) and (141.22,205.51) .. (137.94,205.51) .. controls (134.67,205.51) and (132.02,202.86) .. (132.02,199.59) -- cycle ;
\draw  [fill={rgb, 255:red, 255; green, 255; blue, 255 }  ,fill opacity=1 ][line width=1.5]  (173.19,100.21) .. controls (173.19,96.94) and (175.84,94.29) .. (179.11,94.29) .. controls (182.38,94.29) and (185.03,96.94) .. (185.03,100.21) .. controls (185.03,103.48) and (182.38,106.13) .. (179.11,106.13) .. controls (175.84,106.13) and (173.19,103.48) .. (173.19,100.21) -- cycle ;
\draw  [fill={rgb, 255:red, 255; green, 255; blue, 255 }  ,fill opacity=1 ][line width=1.5]  (173.19,182.54) .. controls (173.19,179.27) and (175.84,176.62) .. (179.11,176.62) .. controls (182.38,176.62) and (185.03,179.27) .. (185.03,182.54) .. controls (185.03,185.81) and (182.38,188.46) .. (179.11,188.46) .. controls (175.84,188.46) and (173.19,185.81) .. (173.19,182.54) -- cycle ;
\draw  [fill={rgb, 255:red, 255; green, 255; blue, 255 }  ,fill opacity=1 ][line width=1.5]  (90.86,100.21) .. controls (90.86,96.94) and (93.51,94.29) .. (96.78,94.29) .. controls (100.05,94.29) and (102.7,96.94) .. (102.7,100.21) .. controls (102.7,103.48) and (100.05,106.13) .. (96.78,106.13) .. controls (93.51,106.13) and (90.86,103.48) .. (90.86,100.21) -- cycle ;
\draw  [fill={rgb, 255:red, 255; green, 255; blue, 255 }  ,fill opacity=1 ][line width=1.5]  (90.86,182.54) .. controls (90.86,179.27) and (93.51,176.62) .. (96.78,176.62) .. controls (100.05,176.62) and (102.7,179.27) .. (102.7,182.54) .. controls (102.7,185.81) and (100.05,188.46) .. (96.78,188.46) .. controls (93.51,188.46) and (90.86,185.81) .. (90.86,182.54) -- cycle ;

\draw (39,44.4) node [anchor=north west][inner sep=0.75pt]    {$\textcolor[rgb]{0.82,0.01,0.11}{f_{i}}$};
\draw (227,45.4) node [anchor=north west][inner sep=0.75pt]    {$\textcolor[rgb]{0.82,0.01,0.11}{f}\textcolor[rgb]{0.82,0.01,0.11}{_{i+1}}$};
\draw (222.31,225.45) node [anchor=north west][inner sep=0.75pt]    {$\textcolor[rgb]{0.82,0.01,0.11}{f}\textcolor[rgb]{0.82,0.01,0.11}{_{i+2}}$};
\draw (34,219.4) node [anchor=north west][inner sep=0.75pt]    {$\textcolor[rgb]{0.82,0.01,0.11}{f}\textcolor[rgb]{0.82,0.01,0.11}{_{j-1}}$};
\draw (120,60.4) node [anchor=north west][inner sep=0.75pt]    {$e_{i}$};
\draw (223,120.4) node [anchor=north west][inner sep=0.75pt]    {$e_{i+1}$};
\draw (141.87,202.99) node [anchor=north west][inner sep=0.75pt]    {$e_{i+2}$};
\draw (57.58,148.45) node [anchor=north west][inner sep=0.75pt]    {$e_{j-1}$};

\end{tikzpicture}
    \caption{Subgraph with balanced white and black vertices enclosed by a cycle}
    \label{fig:distributive_pf_contradiction}
    \end{figure}

    Claim that for any $m$-dimer cover $M$, we must have $\text{mult}_M(e_i) = ... = \text{mult}_M(e_{j-1}) = 0$. To see this, note that for each white vertex $v$, its neighboring edges must be entirely enclosed by $C$. Thus, for any $m$-dimer cover $M$, if we sum up the multiplicities on all edges incident to all white vertices inside $C$, we would get $m\cdot W_C$. However, this is also equal to $m\cdot B_C$, which means that the multiplicities on edges incident to black vertices enclosed by $C$ are fully concentrated on edges enclosed by $C$. Therefore, $\text{mult}_M(e_i) = ... = \text{mult}_M(e_{j-1}) = 0$. However, this contradicts the assumption of $(*)_m$. 
\end{proof}

\begin{cor}
    $M_1\wedge M_2$ (or $M_1\vee M_2$) is the largest (respectively, smallest) $m$-dimer cover that is smaller than (respectively, larger than) both of $M_1, M_2$ with respect to the up-flip poset.
\end{cor}
\begin{proof}
    By Lemma~\ref{lem: compatibility_with_height_posets}, we have $M_1\wedge M_2\leq M_1, M_2$ and $M_1\vee M_2\geq M_1, M_2$. 

    On the other hand, let $M$ be any $m$-dimer cover such that $M\leq M_1,M_2$ (or $M\geq M_1,M_2$, respectively). Then, by Lemma~\ref{lem:height_diff}, $h_M\leq h_{M_1}, h_{M_2}$ (or $h_M\geq h_{M_1}, h_{M_2}$, respectively). Thus, $h_M\leq h_{M_1\wedge M_2}$ (or $h_M\geq h_{M_1\vee M_2}$, respectively). By Lemma~\ref{lem: compatibility_with_height_posets}, we have $M_1\wedge M_2\geq M$ (or $M_1\vee M_2\leq M$, respectively).
\end{proof}

Since taking minimum and maximum of height functions are distributive, we have proved Theorem~\ref{thm: dist_lattice_of_m}. 

Let $\hat{0}_{\mathcal{D}_{m}(G)}$ denote the minimal element in the $m$-dimer lattice of $G$. From now on, we fix our ground cover to be $M_0 := \hat{0}_{\mathcal{D}_{m}(G)}$ and define our height functions with respect to such an $M_0$. 

\begin{defn}
    Let $G$ be a graph with property $(*)_m$. Associate a variable $y_f$ to each non-infinite face $f$ of $G$. Let $M$ be an $m$-dimer cover of $G$. Define the \textbf{height monomial} of $M$ to be $\displaystyle\prod_{f\in F}y^{h_M(f)}_f$. Define the \textbf{$m$-dimer face polynomial} of $G$ to be $$D_{G,m}(\mathbf{y}):= \sum_{M\in \mathcal{D}_m(G)}\displaystyle\prod_{f\in F}y^{h_M(f)}_f$$
\end{defn}

\begin{rem}
    Because any $m$-dimer cover $M$ satisfy that $M\geq M_0$, we have, by Lem~\ref{lem:height_diff}, $h_M(f)\geq h_{M_0}(f) = 0$ for all $f\in F(G)$. Further, by Lem~\ref{lem:height_diff}, the height monomial for $M$ encodes exactly the faces that we need to perform up-flips on $M_0$ and the number of times each face is up-flipped to get to $M$. 
\end{rem}

\begin{defn}
    Let $q$ be a variable, $G$ be any planar bipartite graph with property $(*)_m$. Let $S\subset \mathcal{D}_{m}(G)$. We define the \textbf{rank generating function of $m$-dimer covers in $S$} to be
    \begin{equation}
        \Omega^{S}_m(G,q):= \sum_{M\in S}q^{\sum_{f\in F(G)}h_M(f)}\\
    \end{equation}
\end{defn}

\begin{rem}\label{rem: rank_gen_equals_face_poly}
    Note that $$D_{G,m}(q,q,...,q) = \Omega^{\mathcal{D}_{m}(G)}_{m}(G,q)$$
\end{rem}
\begin{defn}
    Given a planar bipartite graphs $G$ with edge weights, That is, a function $\omega: E(G)\rightarrow \mathbb{R}^+$, define the \textbf{$m$-dimer partition function} to be $$\mathcal{Z}_{G,m}(\omega) := \sum_{M\in \mathcal{D}_m(G)}\prod_{e\in E}\omega(e)^{\text{mult}_M(e)}$$

    We also define the weight of an $m$-dimer cover $M$ to be $\omega(M):= \displaystyle\prod_{e\in E}\omega(e)^{\text{mult}_M(e)}$.
\end{defn}

\begin{defn}\label{def: face_alt_prod}
    Given a planar bipartite graphs $G$ with edge weights $\omega$. Let $f$ be a non-infinite face of $G$, with boundary edges $e_1,e_2,...,e_{2n}$ oriented clockwise, such that $e_{2i}$ are oriented from white to black, and $e_{2i-1}$ are oriented from black to white. Define the \textbf{face alternating product around $f$} to be $$\hat{y}_f := \frac{\prod^{n}_{i=1}\omega(e_{2i-1})}{\prod^{n}_{i=1}\omega(e_{2i})}$$
\end{defn}

\begin{lem}\label{lem: dimer_face_to_partition}
    Given a planar bipartite graphs $G$ with edge weights $\omega$, we have 
    \begin{equation*}
        \mathcal{Z}_{G,m} = \omega(\hat{0}_{\mathcal{D}_m(G)})\cdot D_{G,m}(\hat{\mathbf{y}}) 
    \end{equation*}
\end{lem}
\begin{proof}
    Consider $M\lessdot M'$, related by an up-flip on the face $f$. By definition of an up-flip, we get $\frac{\omega(M')}{\omega(M)} = \frac{\prod^{n}_{i=1}\omega(e_{2i-1})}{\prod^{n}_{i=1}\omega(e_{2i})} = \hat{y}_f$. Iteratively applying Lemma~\ref{lem:height_diff} gives us $\frac{\omega(M)}{\omega(\hat{0}_{\mathcal{D}_{m}(G)})} = \prod_f \hat{y}_f^{h_M(f)}$.
\end{proof}

\subsection{Osvienko's Weights on Snake Graphs}\label{subsec: SW-weights}

\begin{defn}\label{defn: Ovsienko_SW_weighting}   
    Let $G$ be the snake graph $\mathcal{G}[a_1,a_2,...,a_n+1]$ (or the band graph $\overline{\mathcal{G}}[a_1,a_2,...,a_{2n-1},a_{2n}]$). For a block corresponding to $a_{2i-1}$, we assign a weight of $q$ to any west or south edge that lies on the boundary, except for the south edge of the first tile of $a_1$. For a block corresponding to $a_{2i}$, we assign a weight of $q^{-1}$ to any west or south edge that lies on the boundary. We call this assignment the \textbf{South-West weighting} on the graph $G$.
\end{defn}

The weighting in a snake graph is first introduced by Ovsienko~\cite{ovsienko2025qrationalsdimers} as a embedding of a snake graph onto a $\mathbb{Z}^2$ grid. While Ovsienko's original application was restricted to snake graphs, we extend this construction to accommodate the periodic boundary conditions of band graphs.

\begin{rem}\label{rem: face_prod_of_SW_weights}
    With respect to the SW-weighting, the face alternating product of any sqaure face of the snake graph $\mathcal{G}[a_1,...,a_{2n}+1]$ is $q$.
\end{rem}
Osvienko~\cite{ovsienko2025qrationalsdimers} proved the following:

\begin{thm}
[\cite{ovsienko2025qrationalsdimers} Theorem 1.3]\label{thm:weighted_snake_graph_counting}
    Let $G=\mathcal{G}[a_1,a_2,...,a_{2n-1},a_{2n}]$ be a snake graph. Assign weights to its south facing and west facing edges in the above way. Then $\mathcal{Z}_{G,1}(\mathcal{G}[a_1,a_2,...,a_{2n-1},a_{2n}])$ is $\mathcal{R}(q)$ times a Laurent monomial in $q$, where $\mathcal{R}(q)\in \mathbb{Z}_{\geq 0}[q]$ is a polynomial computed by the matrix product:
    \begin{equation}
        \begin{split}\label{eqn:simple_mtx_prod}
            \m{q\mathcal{R}(q)& \Tilde{\mathcal{R}}(q)\\q\mathcal{S}(q)&\Tilde{\mathcal{S}}(q)}&= \m{q& 1\\0&1}^{a_1}\m{q& 0\\q&1}^{a_2}...\m{q& 1\\0&1}^{a_{2n-1}}\m{q& 0\\q&1}^{a_{2n}}\\
        \end{split}
    \end{equation}    
\end{thm}

Our main result generalizes this to the setting of band graphs as follows.

\begin{defn}
    Let a band graph $\overline{\calG}$ be obtained from a snake graph $\calG$. If a perfect matching $P$ of the snake graph $\calG$ satisfy that at least one of the two edges: $(ab)$ from the first tile and $(a'b')$ from the last tile, is contained in $P$, we induce a \textbf{good matching} (or \textbf{good dimer cover}) on the band graph $\overline{\mathcal{G}}$ by removing one of these two edges from $P$. If $P$ contains both such edges, we induce a good matching by including the glued edge $(ab)$ in addition to all the other edges in $P$. We denote the set of such good matchings as  $M(\overline{\mathcal{G}})$.
\end{defn}

\begin{rem}
    From the viewpoint of cluster algebras from surfaces, snake graphs provide a combinatorial expression for arcs (curves with endpoints), while band graphs provide a combinatorial expression for closed curves. The sum of perfect matchings or good matchings yields the elements corresponding to these arcs or closed curves in the cluster algebra~\cite{musiker2011positivity,musiker2013bases}. Specifically, each perfect matching in a snake graph corresponds to a term in the expansion of an arc, and each good matching in a band graph corresponds to a term in the expansion of a closed curve.
\end{rem}

\begin{exmp}
    See Figure~\ref{fig:matching_on_band_graph}. The middle diagram is a perfect matching of the snake graph $\mathcal{G}[1,2,2,3]$, which restricts to a collection of edges on the band graph $\overline{\mathcal{G}}[1,2,2,2]$ covering each vertex exactly once in the leftmost figure, where we used a single green dot to denote the covering of the glued counterpart of that vertex. The right diagram is a collection of edges on $\overline{\mathcal{G}}[1,2,2,2]$ covering each vertex once but it does not come from any perfect matching of $\mathcal{G}[1,2,2,3]$.

    \begin{figure}[h!]
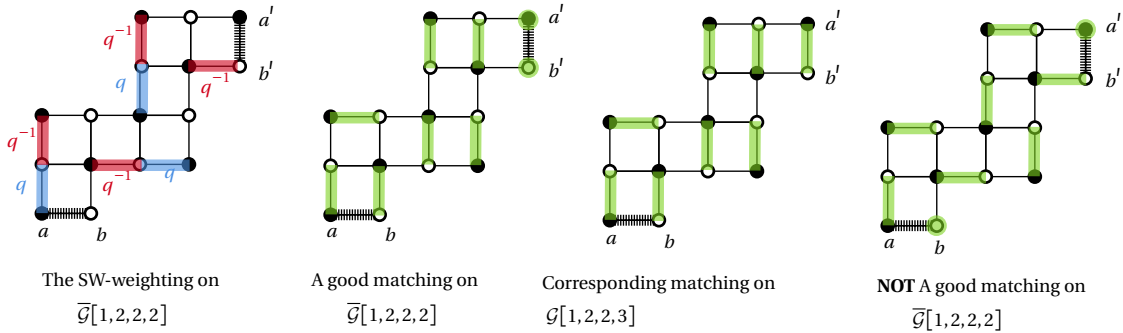

    \centering
    \includestandalone[mode=tex, width=0.9\textwidth]{figures/matching_on_band_graph}
    \caption{An illustration of good (or not good) matchings on band graphs}
    \label{fig:matching_on_band_graph}
    \end{figure}    
\end{exmp}

We assign weights to edges of a band graph. Take $\overline{G}:=\overline{\mathcal{G}}[a_1,a_2,...,a_{2n-1},a_{2n}]$ to be a band graph, and $G:=\mathcal{G}[a_1,a_2,...,a_{2n-1},a_{2n}+1]$ to be the snake graph which glues to form $\overline{G}$. Assign weights to $G$ according to  then via gluing we get a weighting on the band graph $\overline{G}$. Let $\overline{\mathcal{Z}}_{\overline{\mathcal{G}},1}$ denote the weighted sum of good matchings on $\overline{G}$ with respect to this weighting. Then, we prove the following as a special case of our main theorem ~\ref{thm: partition_fcn_is_trace}.

\begin{thm}\label{thm:weighted_band_graph_count}
    Let $\overline{\mathcal{G}}[a_1,a_2,...,a_{2n -1},a_{2n}-1]$ be a band graph, with respect to the above weighting. Then, $\overline{\mathcal{Z}}_{\overline{\mathcal{G}},1}$ is the trace $q\mathcal{R}(q)+\Tilde{\mathcal{S}}(q)$ of the following matrix product, times a Laurent monomial in $q$:
    \begin{equation}
        \begin{split}\label{eqn:simple_mtx_prod_band}
            \m{q\mathcal{R}(q)& \Tilde{\mathcal{R}}(q)\\q\mathcal{S}(q)&\Tilde{\mathcal{S}}(q)}&= \m{q& 1\\0&1}^{a_1}\m{q& 0\\q&1}^{a_2}...\m{q& 1\\0&1}^{a_{2n-1}}\m{q& 0\\q&1}^{a_{2n}}\\
        \end{split}
    \end{equation}
\end{thm}

\begin{rem}
    
The motivation for distinguishing good matchings arises from the geometric interpretation of band graphs as closed loops on a surface. Within the theory of cluster algebras, while arcs correspond directly to cluster variables, closed loops represent more complex algebraic elements—often referred to as ``bracelets'' or elements of the skein algebra. Good matchings is also essential for establishing expansion formulas involving the traces of matrix products~\cite{ezgieminecluster2024}.

\end{rem}

\section{Good Higher Dimer Covers on Band Graphs}\label{sec:good_higher_dimer}

In ~\cite{Musiker2023HigherDC}, the authors computed the number of higher dimer covers on a snake graph by a product of $SL_{m+1}$-matrices ($m\geq 1$) which generalized the matrices $\m{a&1\\1&0}$ coming from continued fractions. In a later work ~\cite{BOSZ26}, the authors defined a $q$-analog of these $SL_{m+1}$-matrices, which can be used to compute the rank generating function for $P$-partitions on a border strip. In this section, we provide a combinatorial interpretation for the trace of the matrix product of the $q$-deformed $SL_{m+1}$-matrices by showing that they compute the dimer partition function on the ``good" $m$-dimer covers on band graphs with respect to ~\cite{ovsienko2025qrationalsdimers}'s weighting rules.

\begin{defn}
Let \(\overline{\mathcal{G}}=\overline{\mathcal{G}}[a_1,a_2,\ldots,a_{2n-1},a_{2n}]\) be a band graph, and let $m\in\mathbb{N}_+$. An $m$-dimer cover
$M$ of $\overline{\mathcal{G}}$ is called a
\textbf{good $m$-dimer cover} if there exist $m$ good matchings $P_1,\ldots,P_m$ of $\overline{\mathcal{G}}$, not necessarily distinct, such that
\[
    \operatorname{mult}_M(e)
    =
    \sum_{i=1}^{m}\operatorname{mult}_{P_i}(e)
\]
for every $e\in E(\overline{\mathcal{G}})$. Here,
$\operatorname{mult}_{P_i}$ is the indicator multiplicity function
of $P_i$; that is,
\[
    \operatorname{mult}_{P_i}(e)
    =
    \begin{cases}
        1, & e\in P_i,\\
        0, & e\notin P_i.
    \end{cases}
\]

We denote the set of all good $m$-dimer covers of
$\overline{\mathcal{G}}$ by \(\overline{\mathcal{D}}_{\overline{\mathcal{G}},m}.\) Let $\omega_{\mathrm{SW}}$ denote the South-West edge weighting
defined in~\Cref{defn: Ovsienko_SW_weighting}. We define the
\textbf{good $m$-dimer partition function} of
$\overline{\mathcal{G}}$ by
\[
    \overline{\mathcal{Z}}_{\overline{\mathcal{G}},m}(q)
    :=
    \sum_{
        M\in
        \overline{\mathcal{D}}_{\overline{\mathcal{G}},m}
    }
    \prod_{e\in E(\overline{\mathcal{G}})}
    \omega_{\mathrm{SW}}(e)^{
        \operatorname{mult}_M(e)
    }.
\]
\end{defn}

\begin{exmp}
    See Figure~\ref{fig:double_dimer_poset_ex} for the flip-poset of good double dimer covers on the band graph $\overline{\mathcal{G}}[2,1]$. Note this is a lattice. 

    \begin{figure}[h!]
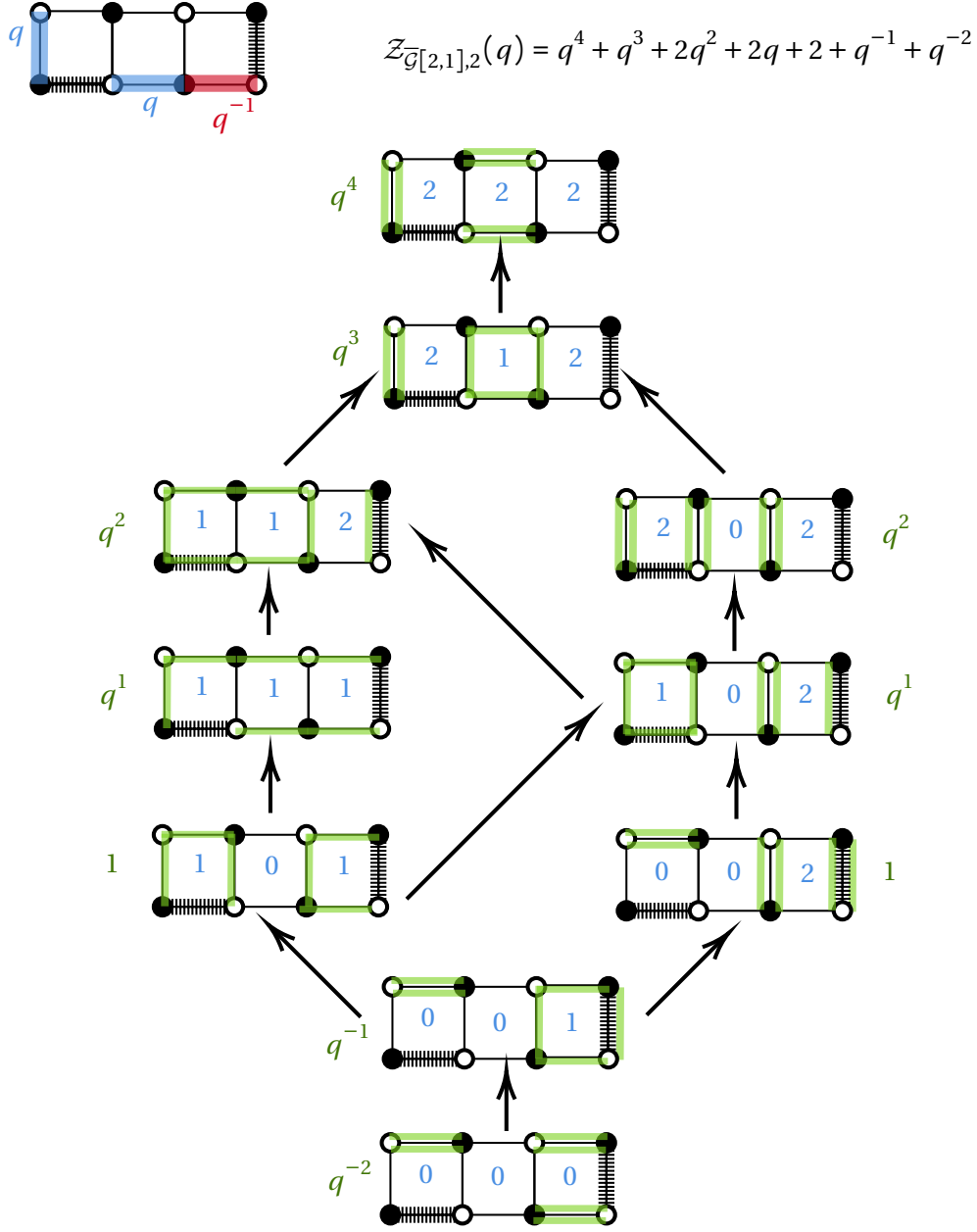

    \centering
    \includestandalone[mode=tex, width=0.8\textwidth]{figures/double_dimer_poset_ex}
    \caption{Lattice of Good Double Dimers}
    \label{fig:double_dimer_poset_ex}
    \end{figure} 
\end{exmp}

Modifying the proof of Lemma~\ref{lem: compatibility_with_height_posets}, we can actually show the set of good $m$-dimer covers form a distributive lattice with respect to up-flips around the square faces of the band graph, where we do not allow up-flips on the "interior face" (the face consisting of all the south and east boundary edges). See Figure~\ref{fig:sw_ne_weighting}.

\begin{prop}\label{prop: lattice_structure_of_good_m_covers}
    The poset generated by up-flips on the square faces of the planar band graph $G:= \overline{\mathcal{G}}[a_1,a_2,...,a_{2n-1},a_{2n}]$ (not allowing flips on the interior face) as a cover relation is a distributive lattice.   
\end{prop}
\begin{proof}
    By Lemma~\ref{lem: property_star_for_planar_bands}, we know that $(\mathcal{D}_{m}(G), \leq)$ is a distributive lattice, where $\leq$ is generated by upflips on any face of $G$ except for its exterior face.
    
    Now take $M_0$ to be any good $m$-dimer cover. We claim that the set of good $m$-dimer covers on $G$ is in bijection with the set of functions on faces with the following property:
    $$\mathcal{H}^{\text{res}}_m(G) = \{h: F(G)\rightarrow \mathbb{Z}| 0\leq h(f_{R(e)}) -h(f_{L(e)})+\text{mult}_{M_0}(e)\leq m, h(f_{\infty}) = h(f_0) = 0 \}$$

    By Lemma~\ref{lem: inequalities_for_heights}, the height function of any $m$-dimer cover of $G$ satisfy the inequalities in the above inequality constraints.
    
    As before, let $e_{a'}, e_{b'}$ be the edges of $G$ in the last tile of the corresponding snake graph incident to $a',b'$ and are not the edge $(a'b')$. Then, note that for any good $m$-dimer cover $M$, we always have $\text{mult}_{M}(e_{a'}) = \text{mult}_{M}(e_{b'})$. On the level of height functions, if we pick a path from $f_\infty$ to $f_0$ passing only through the face of the last tile, with oriented dual edges $\Vec{e^*_{a'}}, \Vec{e^*_{b'}}$, then, we have 
    \begin{equation}
        \begin{split}
            h_M(f_0) - h_M(f_\infty)&= \text{sgn}(\Vec{e^*_{a'}})(\text{mult}_{M}(e_{a'})-\text{mult}_{M_0}(e_{a'}))+ \text{sgn}(\Vec{e^*_{b'}})(\text{mult}_{M}(e_{b'})-\text{mult}_{M_0}(e_{b'}))\\
            &= -(\Vec{e^*_{a'}})(\text{mult}_{M}(e_{a'})-\text{mult}_{M_0}(e_{a'}))+ +(\Vec{e^*_{b'}})(\text{mult}_{M}(e_{b'})-\text{mult}_{M_0}(e_{b'}))\\
            &= (\text{mult}_{M}(e_{b'})-\text{mult}_{M}(e_{a'})) - (\text{mult}_{M_0}(e_{b'})-\text{mult}_{M_0}(e_{a'}))=0\\
        \end{split}
    \end{equation}

    Thus, every good $m$-dimer cover $M$ satisfy that $h_M(f_0) = 0$. Conversely, if a height function satisfy $h_M(f_0) = 0$, we can deduce from the above calculation that $\text{mult}_{M}(e_{a'}) = \text{mult}_{M}(e_{b'})$. 

    Thus, since taking face-wise maximum or minimum preserves the condition $h(f_0) = 0$, we have that if $M_1, M_2$ are good $m$-dimer covers, then $M_1\vee M_2, M_1\wedge M_2$ are also good $m$-dimer covers.

    Finally, we show that the cover relation of $\leq$ when restricted on the set of good $m$-dimer covers are upflips on faces other than the interior face. Let $M_1, M_2$ be two good $m$-dimer covers, with $h_{M_1}(f)\geq h_{M_2}(f)$. By Lemma~\ref{lem: compatibility_with_height_posets}, we know there is a sequence of up-flips to transform $M_1$ to be $M_2$. On the other hand, by Lemma~\ref{lem:height_diff}, we know any up-flip cannot happen around the interior face, since $h_{M_1}(f_0) = 0 = h_{M_2}(f_0)$. 
\end{proof}

We adopt the standard notation for $q$-integers, $q$-factorials, and $q$-binomial coefficients:
\[
[n]_q = 1 + q + \dots + q^{n-1}, \quad [n]_q! = [n]_q [n-1]_q \dots [1]_q, \quad \binom{n}{k}_q = \frac{[n]_q!}{[k]_q! [n-k]_q!}.
\]
The $q$-multiset coefficient is defined as $\multiset{n}{k}_q := \binom{n+k-1}{k}_q$.

Let $R_m$ and $L_m$ denote the upper and lower triangular matrices of $(m+1) \times (m+1)$, respectively, with the entries defined by $(R_m)_{ij} = 1$ for $i \leq j$ and $(L_m)_{ij} = 1$ for $i \geq j$. Furthermore, let $Q_m = \text{diag}(q^m, q^{m-1}, \dots, q, 1)$ be the diagonal matrix of $q$-powers. We define the $q$-deformations of these matrices as:

\[
R_{m}(q):=R_m Q_m= \m{q^m & q^{m-1} & \cdots & q & 1 \\
    0 & q^{m-1} & \cdots & q & 1 \\
    \vdots & & \ddots & \vdots & \vdots \\
    0 & 0 & \cdots & q & 1 \\
    0 & 0 & \cdots & 0 & 1 }, \qquad 
L_{m}(q) :=L_m Q_m= \m{q^m & 0 & \cdots & 0 & 0 \\
    q^m & q^{m-1} & \cdots & 0 & 0 \\
    \vdots & \vdots & \ddots & \vdots & \vdots \\
    q^m & q^{m-1} & \cdots & q & 0 \\
    q^m & q^{m-1} & \cdots & q & 1}.
\]

\begin{lem}[{\cite[Lemma 3.2]{BOSZ26}}]\label{lem:matrix_entries}
    For $j = i+k$ with $k \geq 0$, the $(i,j)$-entry of $R_m(q)^a$ is given by:
    \[
    (R_m(q)^a)_{ij} = q^{a(m+1-i-k)} \multiset{a}{k}_q.
    \]
    Similarly, the entries of the lower triangular powers are given by:
    \[
    (L_m(q)^a)_{ji} = q^{i-j} (R_m(q)^a)_{ij}.
    \]
\end{lem}

Let $W_m$ be the $(m+1) \times (m+1)$ anti-diagonal matrix defined by:
\[
W_m := \begin{pmatrix}
    0 & 0 & \cdots & 0 & 1 \\
    0 & 0 & \cdots & 1 & 0 \\
    \vdots & \vdots & \iddots & \vdots & \vdots \\
    0 & 1 & \cdots & 0 & 0 \\
    1 & 0 & \cdots & 0 & 0
\end{pmatrix}.
\]
We then define the modified matrices $\Lambda_m^+(q,a) := R_m(q)^a W_m$ and $\Lambda_m^-(q,a) := W_m L_m(q)^a$.

\begin{defn}\label{def:graph_matrix}
    For a sequence of nonnegative integers \(\mathbf{a}=[a_1, \dots, a_n]\), the corresponding matrix product $X^{(m)}_{\mathbf{a}}(q)$ is defined as:
    \[
    X^{(m)}_{\mathbf{a}}(q) = \Lambda_m^+(q,a_1) \Lambda_m^-(q,a_2) \dots \Lambda_m^{(-1)^{n-1}}(q,a_n),
    \]
    where the sign of the final factor is determined by the parity of $n$.
\end{defn}

\begin{rem}\label{rem: convert_product}
    For $m=1$, the matrix product in Definition~\ref{def:graph_matrix} is the same as the matrix product formula in (\ref{eqn:simple_mtx_prod}), because $$\m{q& 1\\0&1}^{a}\m{0&1\\1&0} = \m{[a]_q&q^a\\1&0}\quad\quad\quad\m{0&1\\1&0}\m{q& 0\\q&1}^{a} = \m{q[a]_q&1\\q^{a_2}&0}$$
\end{rem}

Combining the theorem from~\cite{BOSZ26} with the section 5 from~\cite{Musiker2023HigherDC} we rewrite the theorem as follows.

\begin{thm}[{\cite[Theorem 3.11]{BOSZ26}}]\label{thm: rank_gen_of_snakes}
    Let $S_{ij}$ be set of $m$-dimer covers of $\mathcal{G}[a_1,a_2,\dots,a_{2n}]$ whose height function values on the first and last tiles are at most $m+1-i$ and $m+1-j$, respectively. Then, the $(i,j)$-entry of the matrix $X^{(m)}_{\mathbf{a}}(q)$ is given by:
    \[
    (X^{(m)}_{\mathbf{a}}(q))_{ij} = 
        q^{m+1-j} \Omega_m^{S_{ij}}(\mathcal{G}[a_1,a_2,...,a_{2n}],q)
    \]
\end{thm}

Our main result is the following:

\begin{thm}\label{thm: partition_fcn_is_trace}
    Let $\overline{\mathcal{G}}[a_1,...,a_{2n}]$ be a planar band graph. Then, $$\overline{\mathcal{Z}}_{\overline{\mathcal{G}},m} = q^{-m\displaystyle\sum^{n}_{i=1}a_{2i}}\text{tr}(X^{(m)}_{\mathbf{a}}(q))$$
\end{thm}

\begin{proof}
    Let $e_a,e_b$ be edges in the last tile which are incident to vertices $a,b$ and are not the gluing edge $(ab)$. Let $e_c = (ab)$, and $e_d$ denote the remaining edge in the last tile. See an example in Figure~\ref{fig:m_dimer_proof}, we have noted in the proof of Proposition~\ref{prop: lattice_structure_of_good_m_covers} that $\text{mult}_{M}(e_a) = \text{mult}_{M}(e_b)$ for any $M\in \overline{\mathcal{D}}_{\overline{\mathcal{G}},m}$.
    
    \begin{figure}[h!]
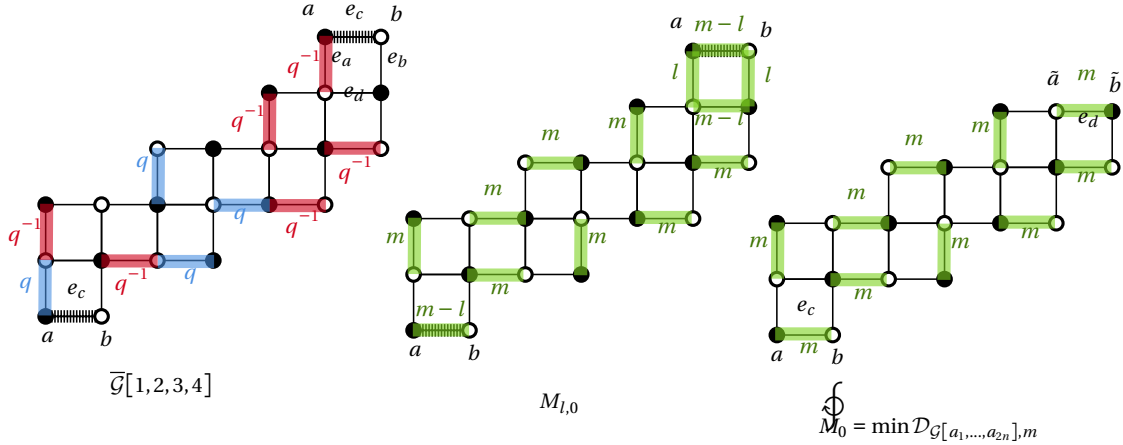

    \centering
    \includestandalone[mode=tex, width=0.9\textwidth]{figures/m_dimer_proof}
    \caption{Left: SW-weighting of a band graph; Middle: edge multiplicities of $M_{l,0}$; Right: minimal $m$-dimer cover on a snake graph}
    \label{fig:m_dimer_proof}
    \end{figure}

    Let $l\in \{0,1,...,m\}$, $\overline{\mathcal{D}}_{\overline{\mathcal{G}},m,l}:=\{M\in \overline{\mathcal{D}}_{\overline{\mathcal{G}}[a_1,...,a_{2n}]} | \text{mult}_M(e_a) = l\}$, and $\overline{\mathcal{Z}}_{\overline{\mathcal{G}},m,l}(q):= \displaystyle\sum_{M\in \overline{\mathcal{D}}_{\overline{\mathcal{G}},m,l}}w(M)$. Then, $\overline{\mathcal{Z}}_{\overline{\mathcal{G}},m}(q) = \sum^m_{l=0}\overline{\mathcal{Z}}_{\overline{\mathcal{G}},m,l}(q)$. Moreover, let $M_{l,0}\in \overline{\mathcal{D}}_{\overline{\mathcal{G}},m,l}$ such that $\text{mult}_{M_{l,0}}(e_c)= \text{mult}_{M_{l,0}}(e_d) = m-l$, and $\text{mult}_{M_{l,0}}(e) = \text{mult}_{M_{0}}(e)$ for all $e\notin \{e_a,e_b,e_c,e_d\}$, where $M_0$ is the unique $m$-dimer cover of $\overline{\mathcal{G}}[a_1,...,a_{2n}]$ whose weight has the lowest power in $q$. See ~\ref{fig:m_dimer_proof} in the middle. Let $\Tilde{M}_0$ be the minimal $m$-dimer cover of the snake graph $\mathcal{G}[a_1,a_2,...,a_{2n}]$ as an induced subgraph of $\overline{\mathcal{G}}[a_1,..a_{2n}]$.
    
    Claim that for any $M\in \overline{\mathcal{D}}_{\overline{\mathcal{G}},m,l}$ there is a unique $M'\in \mathcal{D}_{\mathcal{G}[a_1,...,a_{2n}],m}$ such that $h_{M'}(f_1)\leq m-l$ and $h_{M'}(f_{a_1+...+a_{2n}-1})\leq m-l$ with \[\text{mult}_{M'}(e) = \text{mult}_{M}(e) - \text{mult}_{M_{l,0}}(e) + \text{mult}_{\Tilde{M}_0}(e)\] for all $e\in E(\mathcal{G}[a_1,...,a_{2n}]).$

    First it is clear that for all $v\in V(\mathcal{G}[a_1,...,a_{2n}])$, we have
    \begin{equation*}
        \begin{split}
            \sum_{e\ni v}\text{mult}_{M}(e) - \text{mult}_{M_{l,0}}(e) + \text{mult}_{\Tilde{M}_0}(e) &= \sum_{e\ni v}\text{mult}_{M}(e) - \sum_{e\ni v}\text{mult}_{M_{l,0}}(e) + \sum_{e\ni v}\text{mult}_{\Tilde{M}_0}(e)\\
            &= \begin{cases}
                (m-l)-(m-l) + m, \:\:\text{if}\:\:v=a,b,\tilde{a},\tilde{b}\\
                m - m + m,\:\:\text{else}\\
            \end{cases}
            = m\\
        \end{split}
    \end{equation*}

    We also need to check $\text{mult}_{M}(e) - \text{mult}_{M_{l,0}}(e) + \text{mult}_{\Tilde{M}_0}(e)\in \{0,...,m\}$ for all $e\in E(\mathcal{G}[a_1,...,a_{2n}])$. This is because $$\text{mult}_{\Tilde{M}_0}(e)- \text{mult}_{M_{l,0}}(e)=\begin{cases}
        l, \:\:\text{if}\:\:e=e_c,e_d\\
       0,\:\:\text{else}\\
    \end{cases}$$
    Thus, $\text{mult}_{M}(e) - \text{mult}_{M_{l,0}}(e) + \text{mult}_{\Tilde{M}_0}(e)\geq 0$. Moreover, since $\text{mult}_{M}(e_a)=\text{mult}_{M}(e_b) = l$, we have $\text{mult}_{M}(e_c)\leq m-l, \text{mult}_{M}(e_d)\leq m-l$. Thus $\text{mult}_{M}(e) - \text{mult}_{M_{l,0}}(e) + \text{mult}_{\Tilde{M}_0}(e)\leq m$, for all $e\in E(\mathcal{G}[a_1,...,a_{2n}])$. Lastly, $\text{mult}_{M}(e_c)\leq m-l, \text{mult}_{M}(e_d)\leq m-l$ also implies that \[h_{M'}(f_1)\leq m-l, h_{M'}(f_{a_1+...+a_{2n}-1})\leq m-l.\] 

    Furthermore, $M\mapsto M'$ is a bijection between $\overline{\mathcal{D}}_{\overline{\mathcal{G}},m,l}$ and the $m$-dimer covers of the induced snake graph $\mathcal{G}[a_1,...,a_{2n}]$ where $h(f_1)\leq m-l, h(f_{a_1+...+a_{2n}-1})\leq m-l$. Therefore, by a similar argument to the proof of Lemma~\ref{lem: dimer_face_to_partition}, we have
    \begin{equation*}
        \begin{split}
            \overline{\mathcal{Z}}_{\overline{\mathcal{G}},m,l}(q) &= \omega(M_{l,0})\cdot D_{\mathcal{G}[a_1,...,a_{2n}],m,l}(\hat{\mathbf{y}})|_{\hat{y}_f = q}\\
            &= q^{-m(a_2+a_4+...+a_{2n} - 1)}q^{-l}\cdot \Omega^{S_{l+1,l+1}}_m(\mathcal{G}[a_1,...,a_{2n}],q)\\
            &= q^{-m(a_2+a_4+...+a_{2n})}[q^{m-l}\cdot \Omega^{S_{l+1,l+1}}_m(\mathcal{G}[a_1,...,a_{2n}],q)]\\
            &= q^{-m(a_2+a_4+...+a_{2n})}(X^{(m)}_{\mathbf{a}}(q))_{l+1,l+1}\\
        \end{split}
    \end{equation*}
    where the second equality is obtained from generalizing Remark~\ref{rem: rank_gen_equals_face_poly} to an arbitrary subset and combining with Remark~\ref{rem: face_prod_of_SW_weights}; the last equality comes from Theorem~\ref{thm: rank_gen_of_snakes}.
\end{proof}

\section{Palindromy of Some $m$-Dimer Partition Functions}\label{sec: palindromy}

\begin{defn}
    A Laurent polynomial $f(q)\in \mathbb{C}[q^{\pm1}]$ is \textbf{palindromic} if $f(q^{-1}) = q^{a}f(q)$, for some $a\in\mathbb{Z}$.
\end{defn}

In other literature, palindromic is sometimes referred as symmetric. In the case of single dimer covers, the authors of ~\cite{fence_posets} have shown that the rank generating function of good matchings on a band graph, interpreted as rank polynomials of circular fence posets~\cite{musiker2013bases}, is symmetric. Thanks to David Speyer, we apply properties of traces to Theorem~\ref{thm: partition_fcn_is_trace} to obtain the following shorter proof.

\begin{prop}[~\cite{fence_posets}, Theorem 5.1]
    $\overline{\mathcal{Z}}_{\overline{\mathcal{G}},1}(q)$ is palindromic for all band graphs $\overline{\mathcal{G}}[a_1,a_2,...,a_{2n-1},a_{2n}]$. 
\end{prop}
\begin{proof}
    By Theorem~\ref{thm: partition_fcn_is_trace}, it suffices to show that $\text{tr}(X^{(m)}_{\mathbf{a}}(q))$ is palindromic. 

    Consider the matrix $X = \m{q-1& 1\\q&1-q}$. For any $q\in \mathbb{R}^+$, we have $\det(X) = -(q^2-q+1)\neq 0$. Moreover, one can check that
    \begin{equation}
        \begin{cases}
            XR_1(q)X^{-1} = L_1(q)\\
            XL_1(q)X^{-1} = R_1(q)\\
        \end{cases}\:\:\text{and}\:\:\begin{cases}
            W_1R_1(q)W_1^{-1} = qL_1(q^{-1})\\
            W_1L_1(q)W_1^{-1} = qR_1(q^{-1})\\
        \end{cases}
    \end{equation}
    Thus, 
    \begin{equation}
        \begin{split}
            \text{tr}(R^{a_1}_1(q)L^{a_2}_1(q)...R^{a_{2n-1}}_1(q)L^{a_{2n}}_1(q))&= \text{tr}(XR^{a_1}_1(q)L^{a_2}_1(q)...R^{a_{2n-1}}_1(q)L^{a_{2n}}_1(q)X^{-1})\\
            &= \text{tr}(L^{a_1}_1(q)R^{a_2}_1(q)...L^{a_{2n-1}}_1(q)R^{a_{2n}}_1(q))\\
            &= \text{tr}(W_1L^{a_1}_1(q)R^{a_2}_1(q)...L^{a_{2n-1}}_1(q)R^{a_{2n}}_1(q)W_1^{-1})\\
            &= q^{a_1+a_2+...+a_{2n}}\text{tr}(R^{a_1}_1(q^{-1})L^{a_2}_1(q^{-1})...R^{a_{2n-1}}_1(q^{-1})L^{a_{2n}}_1(q^{-1}))\\
        \end{split}
    \end{equation}
\end{proof}

However, this fails in the case of double dimer covers for general band graphs. For example, the generating function of good double covers $m=2$ on the band graph $\overline{\mathcal{G}}[1,2,2,1]$ is \(1+2q+5q^2+7q^3+11q^4+14q^5+15q^6+13q^7+12q^8+7q^9+5q^{10}+2q^{11}+q^{12}\) which is not palindromic.

Instead, we prove that for a certain class of band graphs $\overline{\mathcal{G}}[a_1,a_2,...,a_{2n-1},a_{2n}]$, the $m$-dimer partition function $\overline{\mathcal{Z}}_{\overline{\mathcal{G}},m}(q)$ as a polynomial in $q$ is palindromic. First, we provide a straightforward linear algebra proof, using properties of traces. Then, we also provide a purely combinatorial proof in the context of dimer theory. 

\begin{defn}
    We say a sequence of positive integers $a_1,...,a_n$ (with $n\geq 2$) is \textbf{almost palindromic} if the subsequence $a_2,...,a_n$ is palindromic.  
\end{defn}

\begin{rem}
    Lagrange is the first mathematician who proved that for any positive integer $d$ which is not a perfect square, $\sqrt{d}$ has the continued fraction expansion $[a_0,\overline{a_1,a_2,...,a_{n-1},2a_0}]$ where $a_i = a_{n-i}$, for all $i=1,...,n-1$~\cite{widz09}, which resembles our ``almost palindromic" condition above.
\end{rem}
\begin{lem}\label{lem: reverse_word}
    \[\text{tr}(X^{(m)}_{\mathbf{a}}(q))=\text{tr}(R_m(q)^{a_{2n}}L_m(q)^{a_{2n-1}}\cdots L_m(q)^{a_1})\]
\end{lem}\
\begin{proof}
    \begin{align*}
        \text{tr}(X^{(m)}_{\mathbf{a}}(q))&=\text{tr}(R_m(q)^{a_1}L_m(q)^{a_2}\cdots L_m(q)^{a_{2n}})\\
        &=\text{tr}((R_m(q)^{a_1}L_m(q)^{a_2}\cdots L_m(q)^{a_{2n}})^T)\\
        &=\text{tr}((L_m(q)^{a_{2n}})^T(R_m(q)^{a_{2n-1}})^T\cdots (R_m(q)^{a_1})^T)\\
        &=\text{tr}(Q_m(R_mQ_m)^{a_{2n}-1}R_mQ_m(L_mQ_m)^{a_{2n-1}-1}L_m\cdots Q_m(L_mQ_m)^{a_1-1}L_m)\\
        &=\text{tr}(Q_mR_m(q)^{a_{2n}}L_m(q)^{a_{2n-1}}\cdots L_m(q)^{a_1-1}L_m)\\
        &=\text{tr}(R_m(q)^{a_{2n}}L_m(q)^{a_{2n-1}}\cdots L_m(q)^{a_1-1}L_mQ_m)\\
        &=\text{tr}(R_m(q)^{a_{2m}}L_m(q)^{a_{2m-1}}\cdots L_m(q)^{a_1})
    \end{align*}

From the third to the last line to the second to the last line, we used the fact $\text{Tr}(ABC)=\text{Tr}(BCA)$.
\end{proof}

\begin{prop}\label{prop: palindromy}
    Let $a_1,a_2,...,a_{2n-1},a_{2n}$ be an almost palindromic sequence, and $\overline{\mathcal{G}}[a_1,a_2,...,a_{2n-1},a_{2n}]$ be a corresponding band graph. Then, $\overline{\mathcal{Z}}_{\overline{\mathcal{G}},m}(q)$ is palindromic. 
\end{prop}

\begin{proof}
    By Theorem~\ref{thm: partition_fcn_is_trace}, it suffices to show that $$R(q):= \text{tr}(X^{(m)}_{\mathbf{a}}(q)) = \text{tr}(R^{a_1}_m(q)L^{a_2}_m(q)...R^{a_{2n-1}}_m(q)L^{a_{2n}}_m(q))$$ is palindromic, for all $m\geq 1$.

    First, note that $W_mR_m(q)W_m = q^{m}L_m(q^{-1})$, and $W_mL_m(q)W_m = q^{m}R_m(q^{-1})$. Thus,
    \begin{equation}
        \begin{split}
            R(q) &=\text{tr}(R^{a_1}_m(q)L^{a_2}_m(q)...R^{a_{2n-1}}_m(q)L^{a_{2n}}_m(q))\\
            &= \text{tr}(W_mR^{a_1}_m(q)L^{a_2}_m(q)...R^{a_{2n-1}}_m(q)L^{a_{2n}}_m(q)W_m)\\
            &= q^{m(a_1+a_2+...+a_{2n-1}+a_{2n})}\text{tr}(L^{a_1}_m(q^{-1})R^{a_2}_m(q^{-1})...L^{a_{2n-1}}_m(q^{-1})R^{a_{2n}}_m(q^{-1}))\\
            &= q^{m(a_1+a_2+...+a_{2n-1}+a_{2n})}\text{tr}(R^{a_2}_m(q^{-1})...L^{a_{2n-1}}_m(q^{-1})R^{a_{2n}}_m(q^{-1})L^{a_1}_m(q^{-1}))\\
            &= q^{m(a_1+a_2+...+a_{2n-1}+a_{2n})}\text{tr}(R^{a_1}_m(q^{-1})L^{a_{2n}}_m(q^{-1})R^{a_{2n-1}}_m(q^{-1})...L^{a_2}_m(q^{-1}))\\
        \end{split}
    \end{equation}
    where the last step follows directly from Lemma~\ref{lem: reverse_word}.

    Lastly, since $(a_1,...,a_{2n})$ is almost palindromic, we have $(a_2,...,a_{2n-1},a_{2n}) = (a_{2n},a_{2n-1},...,a_2)$. Thus, 
    \begin{equation}
        \begin{split}
            \text{tr}(R^{a_1}_m(q^{-1})L^{a_{2n}}_m(q^{-1})R^{a_{2n-1}}_m(q^{-1})...L^{a_2}_m(q^{-1}))&= \text{tr}(R^{a_1}_m(q^{-1})L^{a_{2}}_m(q^{-1})...R^{a_{2n-1}}_m(q^{-1})L^{a_{2n}}_m(q^{-1})) = R(q^{-1})\\
            \implies R(q) &= q^{m(a_1+a_2+...+a_{2n-1}+a_{2n})}R(q^{-1})\\
        \end{split}
    \end{equation}
\end{proof}

We next present a combinatorial proof of Proposition~\ref{prop: palindromy}. First, we set up more background on the theory of dimer models. 

\begin{defn}
    Given a planar bipartite graph $G$ with edge weights $\omega: E(G)\rightarrow \mathbb{R}^+$. A \textbf{gauge transformation} at vertex $v\in V(G)$ is an operation to obtain a new edge weighting $\omega'$, by multiplying the weights of all edges incident to $v$ simultaneously by some $\lambda\in \mathbb{R}^+$. We say two edge weights $\omega,\omega': E(G)\rightarrow \mathbb{R}^+$ are \textbf{gauge equivalent} if $\omega'$ can be obtained from $\omega$ via a sequence of gauge transformations.
\end{defn}

A practical way to detect gauge equivalence is by checking the face alternating products, defined in ~\ref{def: face_alt_prod}. 

\begin{lem}[{\cite[Section 3.2]{kenyon2009lecturesdimers}}]
    Let $G$ be a finite planar bipartite graph whose (non-infinite) faces are all discs. Then two weights $\omega,\omega'$ are gauge equivalent if and only if the face alternating products around any face $f$ are the same, i.e., $\hat{y}^{\omega}_f = \hat{y}^{\omega'}_f$.
\end{lem}

\begin{rem}
    Note that if two weights $\omega,\omega'$ are gauge equivalent, then the dimer partition functions computed with respect to them are scalar multiples of each other. 
\end{rem}

The key idea of the combinatorial proof is that, we show that Ovsienko's SW weighting is in a broader sense "gauge equivalent" to a NE weighting, and the dimer partition function computed with respect to the north-east weighting is equal to substituting $q\mapsto q^{-1}$ in the original dimer partition functions. Moreover, each gauge transformation only multiplies on each vertex by $q^{a}$ for some $a\in \mathbb{Z}$. 

\begin{defn}\label{defn: NE_weighting}
    Embed the underlying snake graph of $\overline{\mathcal{G}}[a_1,a_2,\dots,a_{2n-1},a_{2n}]$ into the $\mathbb{Z}^2$ grid such that the lower-left vertex of the first tile is located at $(0,0)$.
    \begin{enumerate}
    \item For the north edges of the tiles that lie on the boundary of the graph, we label the edge in the first column with $q$ if its endpoints are at $(0,2i)$ and $(1,2i)$, and with $q^{-1}$ if its endpoints are at $(0,2i+1)$ and $(1,2i+1)$, subsequently alternating the labels between $q$ and $q^{-1}$ for the remaining boundary edges.
    \item For the east edges of the tiles that lie on the boundary of the graph, we label the edge in the first row with $q$ if its endpoints are at $(2i+1,0)$ and $(2i+1,1)$, and with $q^{-1}$ if its endpoints are at $(2i,0)$ and $(2i,1)$, subsequently alternating the labels between $q$ and $q^{-1}$ for the remaining boundary edges.
    \end{enumerate}
    Thus we assigned weights to the north-east edges of the underlying snake graph, and gluing the end-edges induces a weighting on the band graph. We call this the \textbf{North-East Weighting} of the band graph.
\end{defn}

\begin{exmp}\label{exmp:not_gauge_equiv_weights}
    Consider the band graph $\overline{\mathcal{G}}[1,2,2,1]$. See Figure~\ref{fig:sw_ne_weighting} left for Osvienko's south-west weighting. See Figure~\ref{fig:sw_ne_weighting} right for the corresponding north-east weighting rules. 

    \begin{figure}[h!]
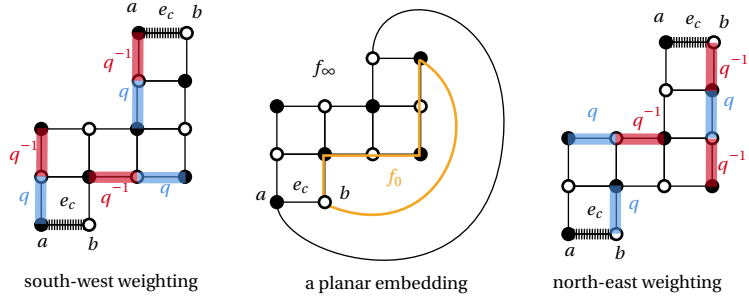

    \centering
    \includestandalone[mode=tex, width=0.6\textwidth]{figures/sw_ne_weighting}
    \caption{SW weighting is not gauge equivalent to NE weighting}
    \label{fig:sw_ne_weighting}
    \end{figure}

    Note that, with respect to the same planar embedding (See Figure~\ref{fig:sw_ne_weighting} in the middle), the south-west weighting is \textbf{not} gauge equivalent to the north-east weighting, because the face alternating product of $f_0$ are different: $q^{-2}$ for the south-west weighting, and $q^{-4}$ for the north-east weighting.
\end{exmp}

However, we will next show that, the $m$-dimer partition functions, when computed with respect to the SW and NE weightings, only differ by a multiple of $q^{a}$ for some $a\in \mathbb{Z}$, via the following special constructions.

\begin{defn}
    Start with a band graph $\overline{\mathcal{G}}[a_1,a_2,...,a_{2n-1},a_{2n}]$, and consider its snake graph counterpart, $G := \mathcal{G}[a_1,a_2,...,a_{2n-1},a_{2n}+1]$. Label the north-west side edges of $G$ in increasing order along the arrows, $e_1,e_2,...,e_{l}$, and label the south-east side edges of $G$ excluding the south-edge of the first tile, by $e_{l+1},...,e_{2(a_1+a_2+...+a_{2n})}$. Take $G'$ to be a copy of $G$, with copied labels $e'_i, i\in [2(a_1+...+a_{2n})]$. Glue $G'$ to $G$ along the last edge $e_c$ such that the edges $e'_{l+1},...,e'_{2(a_1+a_2+...+a_{2n})}$of $G'$ are on the same side as $e_1,...,e_{l}$ of $G$. Lastly, glue the south edge of the first tile of $G$ to the last edge $e'_c$ of $G'$. We call the resulting graph the \textbf{doubling graph} of $\overline{\mathcal{G}}[a_1,a_2,...,a_{2n-1},a_{2n}]$, and denote it as $\overline{\mathcal{G}}^2[a_1,a_2,...,a_{2n-1},a_{2n}]$.
\end{defn}

\begin{exmp}\label{exmp: doubling_graph}
Continuing with Example~\ref{exmp:not_gauge_equiv_weights}. We construct the doubling graph $\overline{\mathcal{G}}^2[1,2,2,1]$ in the middle diagram of Figure~\ref{fig:doubling_graph_ex}.
    \begin{figure}[h!]
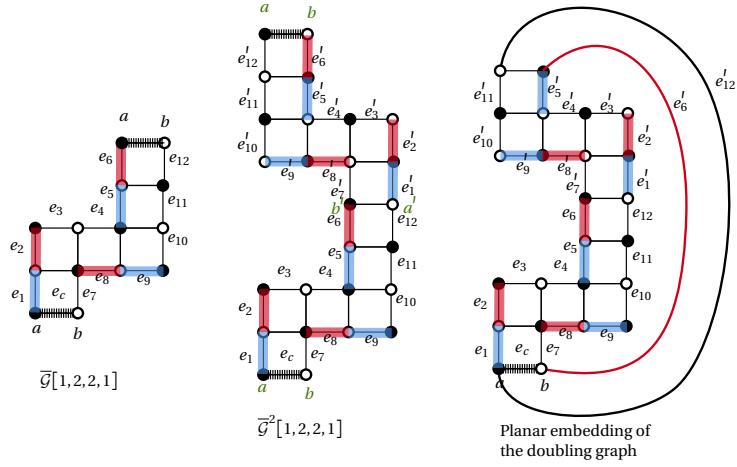

    \centering
    \includestandalone[mode=tex, width=0.6\textwidth]{figures/doubling_graph_ex}
    \caption{Constructing the doubled graph from a band graph}
    \label{fig:doubling_graph_ex}
    \end{figure}
\end{exmp}

\begin{rem}\label{rem: pairing_rules}
    The doubling graph constructed in such a way has a natural pairing of its boundary edges, $e_i\sim e'_i$. Each pair of edges $e_i\sim e'_i$ has the same edge weights. Moreover, such pairings also extend naturally to a pairing of all edges and all vertices. For example, the vertex $a$ of the first tile is paired up with $a'$ in the $(a_1+...+a_{2n})$-th tile in $\overline{\mathcal{G}}^2[a_1,...,a_{2n}]$. One can see that the edges incident to $a,a'$ are also in pairs with each other. For example, in Figure~\ref{fig:doubling_graph_ex}, the neighboring edges of $a,a'$ are $e_1,e_{12}, (ab)$. We have $e'_1,e'_{12}, (a'b')$, and $e_1\sim e_1',e_{12}\sim e'_{12}, (ab)\sim (a'b')$. Note, however, each pair of vertices have opposite colors.  
\end{rem}

\begin{defn}
    Let $v,v'$ be a pair of vertices of $\overline{\mathcal{G}}^2[a_1,...,a_{2n}]$ under the above pairing rules in Remark~\ref{rem: pairing_rules}. Let $\lambda\in\mathbb{R}^+$. A \textbf{doubled gauge transformation} at $(v,v')$ is the change of edge weights $w(e)\mapsto\lambda w(e), w(e')\mapsto\lambda w(e')$ for all edges $e\sim v, e'\sim v'$. 
\end{defn}

\begin{rem}\label{rem:embedding_of_dimers}
    Note that any good $m$-dimer cover $M$ on $\overline{\mathcal{G}}[a_1,...,a_{2n}]$ extends to a $m$-dimer cover $M^2$ on $\overline{\mathcal{G}}^2[a_1,...,a_{2n}]$, by defining $\text{mult}_{M^2}(e') := \text{mult}_{M}(e)$ for each pair of vertices $e\sim e'$. With respect to an edge weighting $w$, we have $w(M^2) = w(M)^2$. 
\end{rem}

\begin{exmp}
    We continued with the doubled graph constructed in Example~\ref{exmp: doubling_graph}. Figure~\ref{fig:doubled_gauge_ex} below shows a sequence of doubled gauge transformations that changed the SW weights of the induced subgraph $\mathcal{G}[a_1,...,a_{2n}+1]$ to the NE side. 
    \begin{figure}[h!]
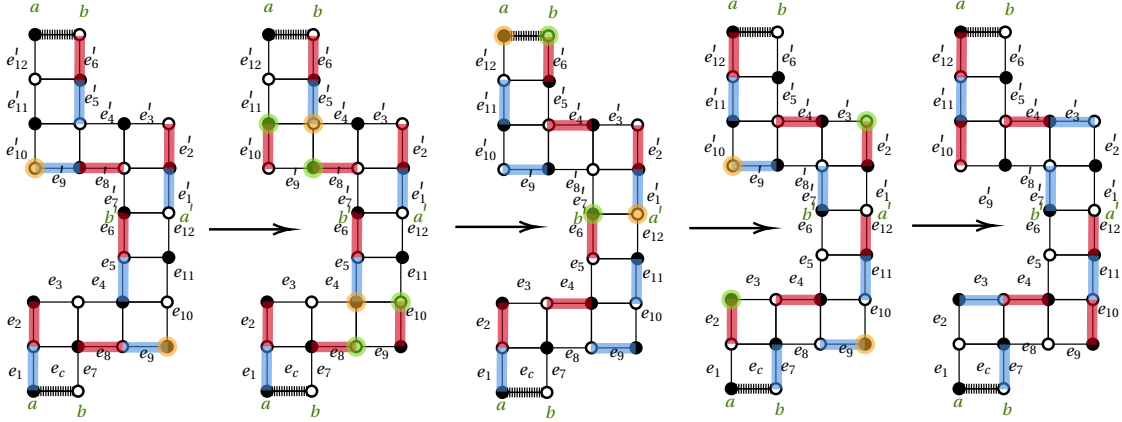

    \centering
    \includestandalone[mode=tex, width=0.9\textwidth]{figures/doubled_gauge}
    \caption{A sequence of doubled gauge transformations. An orange circle on a vertex corresponds to $\lambda = q^{-1}$, and a green circle on a vertex corrresponds to $\lambda = q$}
    \label{fig:doubled_gauge_ex}
    \end{figure}
\end{exmp}

We next show there is always such a sequence of doubled gauge transformations to "move" the SW weights to the NE side.

\begin{lem}\label{lem: conditions_for_double_gauge}
    Fix $q\in \mathbb{R}^+$. Let $G_1$ be a planar bipartite graph considered in the ambient space $\mathbb{S}^2$ where antipetal vertices have opposite colors, and $G_0$ be a graph embedded in $\mathbb{R}\mathbb{P}^2$. Let $\pi: G_1\rightarrow G_0$ be a degree $2$ covering map by identifying antipetal vertices of $G_1$, $v^\circ_j\sim v^\bullet_j$. Let $w_1,w_2$ be edge weights on $G_1$ such that $w_i((v^\circ_jv^\bullet_{j'}))=w_i((v^\bullet_jv^\circ_{j'})) = q^{a_{i,j,j'}}$ for some $a_{i,j,j'}\in\mathbb{Z}, i=1,2$. Suppose further that around each face $F$ of $G_1$, the face alternating products of $w_1,w_2$ are the same. Then, there exists a sequence of pairs of vertices $(v^\circ_{j_1}, v^\bullet_{j_1}),...,(v^\circ_{j_k}, v^\bullet_{j_k})$, and constants $\lambda_1=q^{n_1},...,\lambda_k = q^{n_k}, n_1,...,n_k\in \frac{1}{2}\mathbb{Z}$, such that after performing gauge transformations on $v^\circ_{j_t}, v^\bullet_{j_t}$ with $\lambda_t$ for all $t=1,...,k$, we can transform $w_1$ to $w_2$. 
\end{lem}

\begin{proof}
    Define $1$-forms on oriented edges of $G_1$, by $$\omega_i(\Vec{e}):=\begin{cases}
        \log_q(w_i(e)), \:\:\text{if}\:\:\Vec{e}=\circ\xrightarrow[]{e}\bullet\\
        -\log_q(w_i(e)),
        \:\:\text{if}\:\:\Vec{e}=\bullet\xrightarrow[]{e}\circ\\
    \end{cases}, i=1,2$$ 
    Since the face alternating products of $w_1,w_2$ are the same around any face of $G_1$, we know $d\omega_1(f) =d\omega_2(f)$. That is, $\omega_1-\omega_2$ is a closed $1$-form. Since $H^1(\mathbb{S}^2;\mathbb{Z})=0$, any closed $1$-form is also exact: there exists $f: V(G_1)\rightarrow \mathbb{Z}$, such that $df = \omega_1-\omega_2$. 

    Since $w_i((v^\circ_jv^\bullet_{j'}))=w_i((v^\bullet_jv^\circ_{j'})), i=1,2$, we have 
    \begin{equation}\label{eqn: special_prop_of_f}
        \begin{split}
            (\omega_1-\omega_2)((v^\circ_jv^\bullet_{j'})) &= (\omega_1-\omega_2)((v^\circ_{j'}v^\bullet_{j}))\\
            \implies f(v^\bullet_{j'})-f(v^\circ_j) &= f(v^\bullet_{j})-f(v^\circ_{j'})
        \end{split}
    \end{equation}
    
    Now, define $g:V(G)\rightarrow \frac{1}{2}\mathbb{Z}$, by $g(v^\circ_j) := \frac{1}{2}(f(v^\circ_j)-f(v^\bullet_j))$, and $g(v^\bullet_j) := \frac{1}{2}(f(v^\bullet_j)-f(v^\circ_j))$. Then,
    \begin{equation*}
        \begin{split}
            dg((v^\circ_jv^\bullet_{j'}))&= g(v^\bullet_{j'})-g(v^\circ_j)\\
            &= \frac{1}{2}(f(v^\bullet_{j'})-f(v^\circ_{j'}))-\frac{1}{2}(f(v^\circ_j)-f(v^\bullet_j))\\
            &= \frac{1}{2}(f(v^\bullet_{j'})-f(v^\circ_j))+\frac{1}{2}(f(v^\bullet_j)-f(v^\circ_{j'}))\\
            &= f(v^\bullet_{j'})-f(v^\circ_j) = df((v^\circ_jv^\bullet_{j'}))=(\omega_1-\omega_2)((v^\circ_jv^\bullet_{j'}))\\
        \end{split}
    \end{equation*}
    where the last line follows from (~\ref{eqn: special_prop_of_f}). Similarly, $dg((v^\bullet_jv^\circ_{j'})) = (\omega_1-\omega_2)((v^\bullet_jv^\circ_{j'}))$. 

    Lastly, we examine the effect of $g$ on the level of edge weights, $w_1,w_2$. By definition, we have $$\log_q(\frac{w_1((v^\circ_jv^\bullet_{j'}))}{w_2((v^\circ_jv^\bullet_{j'}))}) = (\omega_1-\omega_2)((v^\circ_jv^\bullet_{j'})) = g(v^\bullet_{j'})-g(v^\circ_j)$$ Thus, $w_2((v^\circ_jv^\bullet_{j'})) = \displaystyle\frac{q^{g(v^\circ_j)}}{q^{g(v^\bullet_{j'})}}w_1((v^\circ_jv^\bullet_{j'}))$. This is saying that $g: V(G)\rightarrow \frac{1}{2}\mathbb{Z}$ corresponds to performing gauge transformation on each white vertex $v^\circ_j$ by $q^{g(v^\circ_j)}$ and on each black vertex $v^\bullet_j$ by $q^{-g(v^\bullet_j)}$. Note that, $g(v^\bullet_j) = \frac{1}{2}(f(v^\bullet_j)-f(v^\circ_j)) = -g(v^\circ_j)$. Thus, on the black vertex $v^\bullet_j$, the multiplicative constant is $q^{-g(v^\bullet_j)} = q^{g(v^\circ_j)}$, the same as the multiplicative constant on its antipetal white vertex $v^\circ_j$.
\end{proof}

\begin{cor}\label{cor:SW_to_NE_weights}
    Let $\overline{\mathcal{G}}^2[a_1,...,a_{2n}]$ be a doubling graph of a band graph. Then there exists a sequence of doubled gauge transformations to change the SW-weighting of the induced snake graph $\mathcal{G}[a_1,...,a_{2n}+1]$ to its NE-weighting, where the last tile contains a $q^{-1}$ weighted edge. 
\end{cor}
\begin{proof}
    By Lemma~\ref{lem: conditions_for_double_gauge}, it suffices to check the face alternating products are the same. Note that the face products for each square face is $q$ for under both the SW-weighting and the NE-weighting. In addition, each outer face has a face product of $q^{-(a_1+...+a_{2n})}$, under both the SW-weighting and the NE-weighting.
\end{proof}

\begin{proof}[Combinatorial Proof of Proposition~\ref{prop: palindromy}]
    Let $\overline{\mathcal{G}}_1:=\overline{\mathcal{G}}[a_1,a_2,...,a_{2n-1},a_{2n}]$, and $\overline{\mathcal{G}}_2:=\overline{\mathcal{G}}[a_1,a_{2n},...,a_{3},a_{2}]$.

For each $q\in\mathbb{R}^+$ fixed, by Remark~\ref{rem:embedding_of_dimers} and Corollary~\ref{cor:SW_to_NE_weights}, there is some $n(q)\in \mathbb{Z}$, such that
\begin{equation}
    \overline{\mathcal{Z}}_{\overline{\mathcal{G}}_1,m}(q) = q^{n(q)}\overline{\mathcal{Z}}_{\overline{\mathcal{G}}_2,m}(q^{-1})
\end{equation}

We show this $n(q)$ does not depend on $q$. We get
\[n(q) = \log_q\left(\frac{\overline{\mathcal{Z}}_{\overline{\mathcal{G}}_1,m}(q)}{\overline{\mathcal{Z}}_{\overline{\mathcal{G}}_2,m}(q^{-1})}\right)=\frac{\ln\left(\frac{\overline{\mathcal{Z}}_{\overline{\mathcal{G}}_1,m}(q)}{\overline{\mathcal{Z}}_{\overline{\mathcal{G}}_2,m}(q^{-1})}\right)}{\ln{q}}\]
where it is defined and continuous on $(0,1)$ and $(1,\infty)$. By applying L'Hôpital's rule, we can show that $n(q)$ can be extended continuously to $q=1$ as well. Therefore, $n(q)$ is continuous on $\mathbb{R}_{\geq0}$. The image of this map $n$ is in $\mathbb{Z}\subset \mathbb{R}$. On the other hand, continuous functions map connected subsets of $\mathbb{R}$ to connected subsets of $\mathbb{R}$. Thus, $n$ must be a constant function.

Lastly, since $(a_1,a_2,...,a_{2n}) = (a_1,a_{2n},...,a_{2})$. We have $\overline{\mathcal{Z}}_{\overline{\mathcal{G}}_1,m}(q) = q^{N}\overline{\mathcal{Z}}_{\overline{\mathcal{G}}_2,m}(q^{-1})$, for all $q\in \mathbb{R}^+$, for some $N\in \mathbb{Z}$.
\end{proof}

\section{$\mathbf{g}$-vector and Dimer Partition Function}\label{sec: connections_to_cluster}

In this section, we explore how the dimer partition function operates within the context of cluster algebras from surfaces. In~\cite{MMSV24}, the authors demonstrate that dimer face polynomials coincide with $F$-polynomials in specific cluster algebras. We will show that the difference in the power of $q$ between the dimer partition function and the rank generating function of snake graphs or band graphs is governed by another essential component of the cluster algebra: the $\mathbf{g}$-vector.

In cluster algebras, the $\mathbf{g}$-vector is an integer vector assigned to each cluster variable that tracks its algebraic weight relative to a chosen initial cluster~\cite{FZ2007}. Formally, when the cluster algebra is endowed with principal coefficients, each cluster variable can be expressed as a multi-homogeneous polynomial; the $\mathbf{g}$-vector is defined as the degree vector of its unique lowest-degree monomial. In the context of cluster algebras arising from surfaces, $\mathbf{g}$-vectors carry a natural geometric interpretation, corresponding to the shear coordinates or intersection numbers of arcs with respect to a fixed triangulation.

In~\cite{banaian2024skein}, the authors gave the construction of $\mathbf{g}$-vectors in posets. Here, we convert this into snake graph language. 

Given a snake graph $\mathcal{G}_\gamma$ associated with an arc or closed curve $\gamma$, we let $\{T_1, \dots, T_d\}$ be the sequence of tiles. Each tile $T_k$ is labeled by an arc $\tau \in \mathcal{T}$. Let a triangulation $\mathcal{T}$ be $\mathcal{T}=\{\tau_1,\ldots\tau_N\}$ and $\{\tau_{i_1},\ldots,\tau_{i_d}\}$ be the arcs in $\mathcal{T}$ that is intersected by $\gamma$ in order. We work in $\mathbb{R}^N$ with the standard basis $\{\mathbf{e}_1, \dots, \mathbf{e}_N\}$ indexed by the arcs of the triangulation.

The components of the $g$-vector $\mathbf{g}_\gamma = -\mathbf{a}_\gamma + \mathbf{b}_\gamma + \mathbf{r}_\gamma$ can be interpreted via the combinatorial properties of $\mathcal{G}_\gamma$ and its distributive lattice of perfect matchings as follows:

\begin{itemize}
    \item \textbf{The vector $\mathbf{a}_\gamma$:} Let $P_{\text{min}}$ be the minimal perfect matching of $\mathcal{G}_\gamma$. We define $\mathbf{a}_\gamma = \sum_{i=1}^n c_i^{\min} \mathbf{e}_i$, where $c_i^{\min}$ is the number of tiles labeled by $\tau_j$ that are flippable in $P_{\text{min}}$.

    \item \textbf{The vector $\mathbf{b}_\gamma$:} Let $P_{\text{max}}$ be the maximal perfect matching of $\mathcal{G}_\gamma$. We define $\mathbf{b}_\gamma = \sum_{i=1}^n c_i^{\max} \mathbf{e}_i$, where $c_i^{\max}$ is the number of tiles labeled by $\tau_j$ that are flippable in $P_{\text{max}}$. In the distributive lattice, these correspond to the elements covered by the maximal matching. We refer to these as ``strict'' maximal elements if the corresponding poset element is covered by at least two elements.

    \item \textbf{The boundary vector $\mathbf{r}_\gamma$:} This vector accounts for the specific geometry of the first and last tiles, $T_1$ and $T_d$, based on the decorations of the curve endpoints. Let $s(\gamma)$ be the marked point where $\gamma$ starts:
    \begin{enumerate}
        \item If the start $s(\gamma)$ is plain, we add $\mathbf{e}_\tau$ to $\mathbf{r}_\gamma$ where $\tau$ is the clockwise neighbor of the internal edge $\tau_{i_1}$ in the first triangle $\Delta_0$.
        \item If $s(\gamma)$ is notched, we add $-\mathbf{e}_{\tau'}$ where $\tau'$ is the counter-clockwise neighbor of $\tau_{i_1}$ in $\Delta_0$.
        \item Similar logic is applied to the last tile $T_d$ relative to the internal edge $\tau_{i_d}$ and the final triangle $\Delta_d$.
    \end{enumerate}
    Boundary arcs of the surface do not contribute to $\mathbf{r}_\gamma$, as they have no associated cluster variables.
\end{itemize}

For any arc or closed curve $\gamma$ with an associated snake graph (or band graph) $\mathcal{G}$, we denote the $g$-vector as:
\[\mathbf{g}_\gamma = -\mathbf{a}_\gamma + \mathbf{b}_\gamma + \mathbf{r}_\gamma.\]

To characterize the power of $q$ that shifts the dimer partition function to the rank generating function, we consider the indices of the flippable tiles in the extremal matchings. Let $\mathcal{A} = \{\alpha_1, \alpha_2, \dots, \alpha_p\}$ be the ordered set of indices $i$ such that $a_i \neq 0$ (tiles flippable in the minimal matching). Similarly, let $\mathcal{B} = \{\beta_1, \beta_2, \dots, \beta_r\}$ be the ordered set of indices $i$ such that $b_i \neq 0$ (tiles flippable in the maximal matching).

Due to the zigzag structure of the fence poset, these indices strictly interlace in one of two configurations:
\begin{enumerate}
    \item $\alpha_1 < \beta_1 < \alpha_2 < \beta_2 < \dots$
    \item $\beta_1 < \alpha_1 < \beta_2 < \alpha_2 < \dots$
\end{enumerate}

We define the exponent $t_\gamma$ as the cumulative sum of the intervals between flippable positions in the following manner:
\[
    t_\gamma = 
    \begin{cases} 
    \sum_{j} (\beta_j - \alpha_j) & \text{if } \alpha_1 < \beta_1, \\
    \sum_{j} (\beta_{j+1} - \alpha_j) & \text{if } \beta_1 < \alpha_1.
    \end{cases}
\]
This value characterizes the weight of the minimal element within the higher dimer $m$-dimer cover $\mathcal{D}_m(G)$.

\begin{prop}
The $q$-weight of the minimal matching $\hat{0}$ in the higher dimer cover $\mathcal{D}_m(G)$, which accounts for the shift between the dimer partition function and the rank generating function of the associated poset, is given by:
\[ \omega(\hat{0}_{\mathcal{D}_m(G)}) = q^{mt_\gamma}. \]
\end{prop}

\begin{proof}
The total weight of the minimal matching is determined by the relative positions of the forced edges. Specifically, $\sum \alpha_j$ corresponds to the total count of $q^{-1}$ contributions, while $\sum \beta_j$ corresponds to the total count of $q$ contributions within the lattice structure. When $\alpha_1 < \beta_1$, the indexing for the $\beta$ terms begins at $j=1$; conversely, if $\beta_1 < \alpha_1$, the indices for the $\beta$ terms are shifted by one to account for the relative offset in the color-grid weighting.
\end{proof}

\section{Future Directions}\label{sec: future}

One direction for future investigation is proving the unimodality of rank generating functions for higher dimer covers on snake graphs~\cite{BOSZ26}. Establishing this unimodality, however, needs a departure from the classical techniques typically used for snake graphs~\cite{OR23}. Unlike the single cover case, the rank generating functions of band graphs associated with higher dimer covers are not palindromic in general. As shown in~\Cref{sec: palindromy}, the generating function of $\overline{\mathcal{G}}[1,2,2,1]$ is not palindromic. This makes traditional methodologies—which rely on the existence of palindromic and unimodal structures—inapplicable to the $m$-dimer cover framework. Consequently, future work will focus on identifying new structural properties to establish unimodality.

We did investigate the rank generating functions of band graph and we found that they are also unimodal and the maximum coefficients lie on the half of the maximum degree of the polynomial. The unimodality of the band graph for single cover is shown in~\cite{OR23,OOR24}.

Another promising research direction involves situating these band graph constructions within other combinatorial frameworks~\cite{BOSZ26}. Specifically, we aim to investigate the correspondence between higher dimer covers on band graphs and alternative models such as dual snake graphs, lattice paths on dual snake graphs, and $P$-partitions.

\printbibliography

\end{document}